\documentclass[12pt]{amsart}

\usepackage{
	amsmath,
 	amsfonts,
  	amssymb,
 	amsthm,
        datetime,
        mathdots
	}

\usepackage{tikz-cd}
\usepackage{tikz}
\usepackage{ytableau}

\usepackage{rotating}

\usetikzlibrary{decorations.pathmorphing, decorations.pathreplacing}
\usetikzlibrary{decorations.pathreplacing,backgrounds,decorations.markings}
\usetikzlibrary{arrows,shapes,positioning}

\usepackage{color}

\usepackage{stmaryrd}
\usepackage[cal=boondoxo]{mathalfa}
\usepackage{mathtools}
\usepackage{bbm}

\usepackage{hyperref}
\usepackage[capitalise]{cleveref}

\usepackage{a4wide}

\usepackage[notref,notcite,final]{showkeys}
\usepackage{enumitem}

\theoremstyle{plain}
\newtheorem{thm}{Theorem}[section]
\newtheorem{cor}[thm]{Corollary}
\newtheorem{lem}[thm]{Lemma}
\newtheorem{prop}[thm]{Proposition}
\newtheorem{conj}[thm]{Conjecture}

\theoremstyle{definition}
\newtheorem{rem}[thm]{Remark}

\newtheorem{ex}[thm]{Example}

\newtheorem*{thmm}{Theorem}

\theoremstyle{definition}
\newtheorem{defn}[thm]{Definition}

\def\makeautorefname#1#2{\expandafter\def\csname#1autorefname\endcsname{#2}}
\makeautorefname{lem}{Lemma}%
\makeautorefname{prop}{Proposition}%
\makeautorefname{rem}{Remark}%
\makeautorefname{section}{Section}%

\DeclareMathOperator{\Id}{Id}
\DeclareMathOperator{\Hom}{Hom}
\DeclareMathOperator{\End}{End}

\DeclareMathOperator{\Mat}{Mat}

\DeclareMathOperator{\pos}{pos}

\DeclareMathOperator{\YH}{YH}

\DeclareMathOperator{\Hec}{H}
\DeclareMathOperator{\TL}{TL}
\DeclareMathOperator{\FTL}{FTL}
\DeclareMathOperator{\CTL}{CTL}
\DeclareMathOperator{\BMW}{BMW}
\DeclareMathOperator{\FBMW}{FBMW}

\DeclareMathOperator{\TB}{TB}
\DeclareMathOperator{\BT}{BT}

\definecolor{myblue}{rgb}{0,.5,1}
\definecolor{mygreen}{rgb}{.3,.75,.1}

\newcommand{\nnfootnote}[1]{%
\begin{NoHyper}
\renewcommand\thefootnote{}\footnote{#1}%
\addtocounter{footnote}{-1}%
\end{NoHyper}
}

\title{Framization of Schur--Weyl duality and Yokonuma--Hecke type algebras}

\author{Abel Lacabanne}
\address{A.L.: Laboratoire de Math{\'e}matiques Blaise Pascal (UMR 6620), Universit{\'e} Clermont Auvergne, Campus Universitaire des C{\'e}zeaux, 3 place Vasarely, 63178 Aubi{\`e}re Cedex, France,\newline \href{http://www.normalesup.org/~lacabanne}{www.normalesup.org/$\sim$lacabanne}}
\email{abel.lacabanne@uca.fr}
\author{Lo\"ic Poulain d'Andecy}
\address{L.P.: Laboratoire de Math{\'e}matiques de Reims (UMR 9008), Universit{\'e} de Reims Champagne-Ardenne, Moulin de la Housse - BP 1039, 51100 Reims, France,\newline \href{http://poulain.perso.math.cnrs.fr/}{http://poulain.perso.math.cnrs.fr/}}
\email{loic.poulain-dandecy@univ-reims.fr}

\begin{document}

\begin{abstract}
We study framizations of algebras through the idea of Schur--Weyl duality. We provide a general setting in which framizations of algebras such as the Yokonuma--Hecke algebra naturally appear and we obtain this way a Schur--Weyl duality for many examples of these algebras which were introduced in the study of knots and links. We thereby provide an interpretation of these algebras from the point of view of representations of quantum groups. In this approach the usual braid groups is replaced by the framed braid groups. This gives a natural procedure to construct framizations of algebras and we discuss in particular a new framized version of the Birman--Murakami--Wenzl algebra. The general setting is also extended to encompass the situation where the usual braid group is replaced by the so-called tied braids algebra, and this allows to collect in our approach even more examples of algebras introduced in the knots and links setting.
\end{abstract}

\nnfootnote{\textit{Mathematics Subject Classification 2020.} 20C08, 17B37, 20F36.}
\nnfootnote{\textit{Keywords.} Schur--Weyl duality, Yokonuma--Hecke algebras, framization, framed braid group, tied braid algebra.}

\maketitle

\setcounter{tocdepth}{2}
\tableofcontents


\section{Introduction}

The Yokonuma--Hecke algebra is a natural generalization of the usual Hecke algebra in the sense that the Hecke algebra originally appeared in the study of the permutation representation of $GL_n(F_q)$ with respect to its Borel subgroup, while the Yokonuma--Hecke algebra plays a similar role when replacing the Borel subgroup by its unipotent radical.

The Yokonuma--Hecke algebra was also studied from the point of view of knots and links, and was used to produce invariants \cite{Juy,JuLa1} generalizing the well-known construction of the Jones and HOMFLY--PT polynomials from the usual Hecke algebra. The precise topological meaning of these invariants in relation with the HOMFLY--PT polynomial was then elucidated \cite{CJKL,PAW}. In this context, the Yokonuma--Hecke algebra is referred to as a ``framization'' of the usual Hecke algebra, and there were subsequent attempts to ``framize'' other known algebras and use them for knot theory \cite{Chl,CPA2,FJL,Gou,GJKL,JuLa}. These framizations are usually defined via generators and relations and it is not always clear what the correct definition should be (\emph{e.g.} the different versions for the Temperley--Lieb algebras \cite{Gou}).

The usual Hecke algebra also appears in another famous context, which is the quantum Schur--Weyl duality, relating it to the representation theory of the quantum groups $U_q(\mathfrak{gl}_N)$ \cite{Jim,Res}. The Schur--Weyl duality completes the picture in which the Jones and HOMFLY--PT polynomials are seen as particular cases of Reshetikhin--Turaev invariants associated to quantum groups. Thus a similar interpretation of the Yokonuma--Hecke algebra in a Schur--Weyl duality with quantum groups seems desirable and natural to expect.

\vskip .2cm
The first goal of this paper is to prove a Schur--Weyl duality statement for the Yokonuma--Hecke algebra, as well as for various related algebras such as framizations of the Temperley--Lieb algebra, and the so-called algebra of braids and ties \cite{AJ-BT,RH11}. We will thus find the precise meaning of all these algebras from the point of view of quantum groups.

The usual Hecke algebra can be seen as a deformation of the group algebra of the symmetric group, and the Yokonuma--Hecke algebra has a similar interpretation in terms of a deformation of the wreath product $(\mathbb{Z}/d\mathbb{Z})^n\rtimes \mathfrak{S}_n$, also known as the complex reflection group $G(d,1,n)$. A Schur--Weyl duality context for this group exists   and a quantization of this duality can be obtained in terms of the Ariki--Koike algebra, which is another deformation of the group $G(d,1,n)$, and in terms of the quantum group $U_q(\mathfrak{gl}_{N_1} \oplus \dots\oplus \mathfrak{gl}_{N_d})$, see \cite{MaSt,SaSh} and references therein. We prove that a Schur--Weyl duality with $U_q(\mathfrak{gl}_{N_1} \oplus \dots\oplus \mathfrak{gl}_{N_d})$ also applies to the Yokonuma--Hecke algebra instead of the Ariki--Koike algebra. Nonetheless, we find that the extension of the Schur--Weyl duality of the Hecke algebra to the Yokonuma--Hecke algebra is easier and more natural than for the Ariki--Koike algebra, and is in fact a particular case of a general procedure developed in this paper. Here, the action of a braid group that factors through the Hecke algebra is replaced by an action of the framed braid group that naturally factors through the Yokonuma--Hecke algebra. 

\vskip .2cm
The second goal of the paper is to provide a general procedure to construct framizations of algebras from the point of view of the Schur--Weyl duality. This general procedure culminates in \cref{thm:action-framed} that contains in particular the following statement.
\begin{thmm}
 We have a morphism of algebras $\Phi\ :\ \Bbbk\mathcal{FB}_{d,n}\rightarrow\End_A(V^{\otimes n})$.
\end{thmm}
Here the space $V=V_1\oplus \dots \oplus V_d$ is a direct sum where each $V_b$ is a representation of a bialgebra $A_b$ with the property that the usual braid group acts on the tensor products $V_b^{\otimes k}$ and centralizes the action of $A_b$. Of course the main example we have in mind is when the algebra $A_b$ is a quantum group acting on a finite-dimensional representation $V_b$. The algebras in Schur--Weyl duality in the theorem are then the tensor product $A=A_1\boxtimes\dots\boxtimes A_d$ and the group algebra $\Bbbk\mathcal{FB}_{d,n}$ of the framed braid group. The particular case factoring through the Yokonuma--Hecke algebra is found when $A_b=U_q(\mathfrak{gl}_{N_b})$ and $V_b$ is the vector representation of dimension $N_b$, so that $V$ is the natural vector representation of dimension $N_1+\dots+N_d$ of $A=U_q(\mathfrak{gl}_{N_1} \oplus \dots\oplus \mathfrak{gl}_{N_d})$.

In this approach, the Yokonuma--Hecke algebra is naturally obtained and we recover as well the natural framization of the Temperley--Lieb algebra. This also allows us to give a natural definition of a framization of the Birman--Murakami--Wenzl algebra, which seems to be new. Any algebra appearing in a Schur--Weyl duality can be framized following this procedure. The one-boundary extension of the previous theorem is also proved, involving the framed affine braid group, and is applied to affine versions of framizations of algebras, such as the affine and cyclotomic Yokonuma--Hecke algebras.

This procedure of framization is mostly used to construct invariants of knots and links, and sometimes several versions of a framization of an algebra are proposed. For example, there have been at least three tentatives of framization of the Temperley--Lieb algebra \cite{Gou}. With our procedure of framization, we find that the correct framization of the Temperley--Lieb algebra should be the Schur--Weyl dual of the quantum group $U_q(\mathfrak{gl}_{2} \oplus \dots\oplus \mathfrak{gl}_{2})$. One of the three proposed framizations is then natural to consider from our point of view, as it was also advocated in \cite{CP}, see also \cref{rem-TL}. 

Another advantage of the approach through the Schur--Weyl duality is that we recover naturally some isomorphism theorems for the framizations of algebras (see \cref{rem:iso_matrix} for example). In fact, we advocate the point of view that  the correct framization of a given algebra should be isomorphic to a direct sum of matrix algebras over tensor products of the algebra we started with. This comes up naturally in the approach through the Schur--Weyl duality and is also natural from the point of view of invariants of knots and links \cite{CP,JPA,PAW}. The representation theory of the framizations of algebras is also recovered from this point of view.

\vskip .2cm
Finally we also provide the general procedure allowing to obtain algebras related to braids and ties. One of our motivations was to explore and to generalize, from the Schur--Weyl duality setting, the relationship between the algebra of braids and ties and the fixed points of the Yokonuma--Hecke algebra under the action of a symmetric group. To do so in a general setting, we first upgrade this relationship to the level of the framed braid group. We define a natural action of the symmetric group on the group algebra of the framed braid group, and by taking the fixed points under this action we make the connection with the tied braid monoid of \cite{AiJ}.

We upgrade the Schur--Weyl duality previsouly obtained by adding a natural symmetry of the representations, so that the fixed point subalgebras for a natural action of the symmetric group on the framizations of algebras naturally enter the picture. The general procedure leads to the following braids and ties version of the preceding theorem (see \cref{thm:BT_centralizer}).
\begin{thmm}
 We have a morphism of algebras $\Psi\ :\ \TB_n \rightarrow \End_{A\rtimes \mathfrak{S}_d}(V^{\otimes n})$.
\end{thmm}
Here the algebra $\TB_n$ is the tied braid algebra while $A$ and $V$ are as in the first theorem, with the additional assumption that all bialgebras $A_b$ are the same and all representations $V_b$ are the same. In this situation, the symmetric group $\mathfrak{S}_d$ naturally acts on $A$ and $V$ naturally becomes a representation of the smash product.

The algebraic description of the centralizers in these Schur--Weyl duality settings will thus provide braids and ties versions of well-known algebras. For example, we recover the algebra of braids and ties from the Yokonuma--Hecke algebra \cite{AJ-BT} and we recover a braids and ties version for the Temperley--Lieb algebra, also called partition Temperley--Lieb algebra \cite{Juy2} (see also \cite{AE} for related constructions). We also propose a definition by generators and relations of a BMW algebra of braids and ties, which seems natural in our approach and which seems different from the one in \cite{AiJ2}.


\vskip .2cm
Our goal is not to be exhaustive concerning the framizations of algebras and their braids and ties versions, but is more to provide a Schur--Weyl duality setting in which these framizations naturally appear. We give many examples, some of those are well known and many others deserve a more thorough study. We also note that the Yokonuma--Hecke algebra appears as a special case in \cite{LNX} where a Schur--Weyl duality statement different from ours is given.

\vskip .2cm
\paragraph{\textbf{Acknowledgements.}} LPA is supported by Agence Nationale de la Recherche Projet AHA ANR-18-CE40-0001 and the international research project AAPT of the CNRS. AL is supported by a PEPS JCJC grant from INSMI (CNRS). The authors thank Catharina Stroppel and Pedro Vaz for useful discussions.

\section{Algebraic preliminaries}
\label{sec:alg_prel}

Let us start with some short and easy algebraic lemmas. Let $\Bbbk$ be a field.

\subsection{External tensor products of algebras.}
\label{sec:cent_tensor}

Let $d\in \mathbb{N}^*$ . For all $1 \leq b \leq d$, let $A_b$ be a unital $\Bbbk$-bialgebra. Recall that this means in particular that we can make tensor products of $A_b$-modules and that we have a notion of a trivial representation $\epsilon_b\ \colon\ A_b\to\Bbbk$, that we denote $\textbf{1}_{A_b}$, and which satisfies $\textbf{1}_{A_b}\otimes V_b\cong V_b\otimes \textbf{1}_{A_b}\cong V_b$ for any $A_b$-module $V_b$, where the isomorphisms are given by the trivial identity map.

We consider the algebra 
\[
A=A_1\boxtimes \cdots \boxtimes A_d\ .
\]
As a vector space, this is the usual tensor product, and the multiplication is performed independently in each factor for pure tensors and extended linearly. Given $A_b$-modules $W_b$ for $b=1,\dots,d$, the tensor product becomes naturally a representation of $A$ that we denote $W_1\boxtimes \cdots \boxtimes W_d$. The tensor product of two such representations of $A$ is defined by performing the tensor product of $A_b$-modules in each factor.

Now we fix an $A_b$-module $V_b$ for each $b=1,\dots,d$. We see it as an $A$-module, namely, we make the following identification:
\begin{equation}\label{identificationVb}
  V_b=\textbf{1}_{A_1}\boxtimes \dots \boxtimes \textbf{1}_{A_{b-1}}\boxtimes V_b \boxtimes \textbf{1}_{A_{b+1}}\boxtimes\dots \boxtimes \textbf{1}_{A_{d}}\ .
\end{equation}
Finally, we define the following $A$-module:
\[
  V=V_1\oplus \dots \oplus V_d\ .
\]
Explicitly, the action of an element $a_1\otimes \cdots \otimes a_d$ in $A$ is given by
\[
  a_1\otimes \cdots \otimes a_d \cdot v = \left(\prod_{c\neq b}\epsilon_{c}(a_{c})\right) a_b\cdot v\,,\quad\forall v\in V_b\ \text{and}\ b=1,\dots,d\ .
\]
We have the following relation between the centralizer of $A$ in $V^{\otimes n}$ and the centralizers of the various $A_b$. 

\begin{lem}
  \label{lem:iso-matrix-alg}
  Suppose that for all $1 \leq b \leq d$ and $r\neq s$ we have $\Hom_{A_b}(V_b^{\otimes r},V_b^{\otimes s})=0$. Then we have
  \begin{equation}\label{iso-matrix-alg}
    \End_{A}(V^{\otimes n}) \simeq \bigoplus_{\nu \vDash_d n}\Mat_{\binom{n}{\nu}}\Bigl(\End_{A_1}(V_1^{\otimes \nu_1})\otimes \cdots \otimes \End_{A_d}(V_d^{\otimes \nu_d})\Bigr)\,,
  \end{equation}
  the sum being taken over all $d$-compositions $\nu=(\nu_1,\ldots,\nu_d)$ of $n$.
\end{lem}

The multinomial coefficients appearing as sizes of the matrix algebras are:
\[
  \binom{n}{\nu}=\frac{n!}{\nu_1!\dots \nu_d!}\ .
\]

\begin{proof}
  We have the following decomposition of the vector space $V^{\otimes n}$:
  \[
    V^{\otimes n}=\bigoplus_{a_1,\dots,a_n=1}^dV_{a_1}\otimes\dots\otimes V_{a_n}\ .
  \]
  Looking at \eqref{identificationVb}, we see that the summand $V_{a_1}\otimes\dots\otimes V_{a_n}$ is an $A$-submodule isomorphic to $V_1^{\otimes \nu_1}\boxtimes \cdots \boxtimes V_d^{\otimes \nu_d}$, where $\nu_b$ is the number of indices among $a_1,\dots,a_n$ which are equal to $b$. There are therefore $\binom{n}{\nu}$ summands corresponding to the composition $\nu$. Therefore we have the following decomposition of $V^{\otimes n}$ as an $A$-module:
  \begin{equation}\label{proof_dec}
    V^{\otimes n} \simeq \bigoplus_{\nu \vDash_d n} \left(V_1^{\otimes \nu_1}\boxtimes \cdots \boxtimes V_d^{\otimes \nu_d}\right)^{\oplus \binom{n}{\nu}}\ .
  \end{equation}
  The statement of the lemma follows from the general fact that given $A_b$-modules $W_b,W'_b$, we have
  \[
    \Hom_A(W_1\boxtimes \cdots \boxtimes W_d,W'_1\boxtimes \cdots \boxtimes W'_d)=\Hom_{A_1}(W_1,W'_1)\otimes \dots\otimes \Hom_{A_d}(W_d,W'_d)\ ,
  \]
  together with the hypothesis which implies that there is no homomorphism commuting with $A$ between summands corresponding to different compositions.
\end{proof}

\begin{rem}\label{rem-assumptionlemma}
  Without the assumption, the isomorphism in the lemma remains valid if we replace the full centralizer $\End_A(V^{\otimes n})$ by its subalgebra generated by the subspaces:
  \[
    \Hom_A(V_{a_1}\otimes\dots\otimes V_{a_d},V_{b_1}\otimes\dots\otimes V_{b_d})\ ,
  \]
  for $a$'s and $b$'s giving the same composition, namely such that $|\{i\ |\ a_i=x\}|=|\{i\ |\ b_i=x\}|$ for all $x=1,\dots,d$.
\end{rem}

\begin{rem}
  One can explicitly give the isomorphism of \cref{lem:iso-matrix-alg} using the idempotents $\pi_{\nu}$ corresponding to the projections on the summands of the decomposition $V^{\otimes n} \simeq \bigoplus_{\nu \vDash_d n} \left(V_1^{\otimes \nu_1}\boxtimes \cdots \boxtimes V_d^{\otimes \nu_d}\right)^{\oplus \binom{n}{\nu}}$. If $\pi_b\colon V \rightarrow V$ is the projection on the summand $V_b$ of $V$, the idempotent $\pi_\nu$ for $\nu\vDash_d n$ is given by
  \[
    \pi_\nu = \sum_{\substack{(b_1,\ldots,b_n)\in\{1,\ldots,d\}^n\\\lvert\{i\mid b_i=k\}\rvert = \nu_k}}\pi_{b_1}\otimes\cdots\otimes\pi_{b_n} \in\End_A(V^{\otimes n}).
  \]
The isomorphism of the lemma then sends $f\in\End_A(V^{\otimes n})$ to the family $(\pi_\nu f \pi_\nu)_{\nu\vDash_d n}$. The assumption of \cref{lem:iso-matrix-alg} indeed implies that if $\nu\neq \mu$ then $\pi_\nu f \pi_\mu = 0$.
\end{rem}

\subsubsection{Consequences.}\label{sec:consequences}

The isomorphism of the lemma implies a Morita equivalence between $\text{End}_A(V^{\otimes n})$ and the direct sum of the algebras inside the matrix algebras. In particular, the irreducible representations of the direct sum in the right hand side of \eqref{iso-matrix-alg} are indexed by
\[
  (\nu,\rho_1,\dots,\rho_d)\ \ \ \ \text{with $\nu \vDash_d n$ and $\rho_b\in\text{Irr}\bigl(\End_{A_b}(V_b^{\otimes \nu_b})\bigr)$,}
\]
the dimension of this representation being $\binom{n}{\nu}\prod_{b=1}^d\dim(\rho_b)$. The total dimension of the algebra is of course:
\[
  \sum_{\nu\vDash_d n}\binom{n}{\nu}^2d^{(1)}_{\nu_1}\dots d^{(d)}_{\nu_d}\,,\ \ \ \ \text{where }d^{(b)}_{\nu_b}=\dim\left(\End_{A_b}(V_b^{\otimes \nu_b})\right).
\]

\subsection{One-boundary extension} 

We keep the same context, and suppose that we are moreover given for all $1 \leq b \leq d$ an $A_b$-module $M_b$. To lighten the notations, we denote:
\[
  C^{(b)}_{n}=\End_{A_b}(M_b\otimes V_b^{\otimes n})\ .
\]
We consider the $A$-module $M=M_1\boxtimes \cdots\boxtimes M_d$ and we have the following one-boundary generalization of \cref{lem:iso-matrix-alg}.

\begin{lem}
  Suppose that for all $1 \leq b \leq d$ and $r\neq s$ we have $\Hom_{A_b}(M_b\otimes V_b^{\otimes r},M_b\otimes V_b^{\otimes s})=0$. Then we have
  \[
    \End_{A}(M\otimes V^{\otimes n}) \simeq \bigoplus_{\nu \vDash_d n}\Mat_{\binom{n}{\nu}}(C^{(1)}_{\nu_1}\otimes \cdots \otimes C^{(d)}_{\nu_d})\ ,
  \]
  the sum being taken over all $d$-compositions $\nu=(\nu_1,\ldots,\nu_n)$ of $n$.
\end{lem}

\begin{proof}
  We use the decomposition \eqref{proof_dec} of $V^{\otimes n}$ in $M\otimes V^{\otimes n}$, and conclude with the same argument as in the proof of \cref{lem:iso-matrix-alg}.
\end{proof}

Similar consequences as in \cref{sec:consequences} can be deduced.

\subsection{Centralizer and group action}
\label{sec:dble_grp}

Let $A$ be a unital $\Bbbk$-algebra. Suppose that we are given a finite group $G$ that acts by algebra automorphisms on $A$. We denote the action $G\times A \rightarrow A$ by $(g,a) \mapsto {}^ga$. We can now define the smash product algebra $A\rtimes G$ of $A$ and $\Bbbk[G]$ as the $\Bbbk$-vector space $A\otimes_{\Bbbk} \Bbbk[G]$ equipped with the following multiplication:
\[
  (a\otimes g) \cdot (b\otimes h) = a\,{}^gb\otimes gh,\ \quad\forall a,b\in A\,,\ \forall g,h\in G.
\]
We will often identify $a\in A$ with $a\otimes 1\in A\rtimes G$ and $g\in G$ with $1\otimes g \in A\rtimes G$.

We now fix an $A\rtimes G$-module $W$. The centralizer algebra $\End_{A}(W)$ inherits of an action of $G$ by conjugation. Namely, this action is defined, for $g\in G$ and $\phi\in \End_{A}(W)$, by
\[
  (g\cdot \phi)(v) = g\cdot \phi(g^{-1}\cdot v),\quad \forall v\in W.
\]
Indeed, we have that $g\cdot \phi\in \End_{A}(W)$ since
\[
  (g\cdot \phi)(a\cdot v)=g\cdot \phi(g^{-1}a\cdot v)=g\cdot \phi({}^{g^{-1}}\!\!\!\!\!a\,g^{-1}\cdot v)=g\,{}^{g^{-1}}\!\!\!\!\!a\cdot \phi(g^{-1}\cdot v)=a\,g\cdot \phi(g^{-1}\cdot v)=a\cdot (g\cdot \phi)(v)\ .
\]

\begin{lem}~
  \label{lem:centralizer_semi}
  \begin{enumerate}[label=(\roman*)]
  \item The centralizer algebra $\End_{A\rtimes G}(W)$ of $A\rtimes G$ is the fixed point subalgebra of $\End_A(W)$: 
    \[
      \End_{A\rtimes G}(W) = \End_A(W)^{G}\ .
    \]
  \item Assume that $\lvert G \rvert$ is invertible in $\Bbbk$. Let $X$ be an algebra equipped with an action of $G$ by algebra automorphisms. If we have a surjective algebra morphism 
    \[
      \varphi\ : X \rightarrow \End_A(W)\,,
    \]
    commuting with the action of $G$, then $\varphi$ restricts to a surjective algebra morphism
    \[
      \varphi\ :\ X^G \rightarrow \End_A(W)^G\ .
    \]
  \end{enumerate}
\end{lem}

\begin{proof}
  For the first item, note that $W$, being an $A\rtimes G$-module, is in particular both an $A$-module and a $G$-module, and that centralizing the action of $A\rtimes G$ is equivalent to centralizing both actions of $A$ and $G$. Then it is immediate that for an element in $\End_A(W)$, centralizing the action of $G$ is equivalent to being in $\End_A(W)^G$.
  
  For the second item, since $\varphi$ commutes with the action of $G$, it is clear that $X^G$ is sent to the fixed points $\End_A(W)^G$. To check the surjectivity, take an element $y$ in $\End_A(W)^G$. From the surjectivity of $\varphi$, there is an element $x$ in $X$ such that $\varphi(x)=y$. Now taking the average $\frac{1}{|G|}\sum_{g\in G}g.x$, it is straightforward to see that this is an element in $X^G$ which is sent to $y$.
\end{proof}

\subsubsection{Main example}

We take the general setting of \cref{sec:cent_tensor} with the additional assumption that all bialgebras $A_1,\dots,A_d$ are equal to one and the same bialgebra $A^{(0)}$, and all modules $V_1,\dots,V_d$ are equal to one and the same $A^{(0)}$-module $V^{(0)}$: 
\[
  A=A^{(0)}\boxtimes \cdots \boxtimes A^{(0)}\ \ \ \ \text{and}\ \ \ \ V_1=\dots=V_d=V^{(0)}\ .
\]
In this case, the algebra $A$ is equipped naturally with an action of the symmetric group $\mathfrak{S}_d$, obtained by permuting the tensorands:
\[
  \sigma \cdot (x_1\otimes \cdots \otimes x_d) = x_{\sigma^{-1}(1)}\otimes \cdots \otimes x_{\sigma^{-1}(d)}\ \ \ \forall \sigma \in \mathfrak{S}_d\,,\ \ \forall x_1,\ldots,x_d\in A^{(0)}\ .
\]
Therefore we consider the algebra $A\rtimes \mathfrak{S}_d$ defined by this action. 

There is also a natural action of $\mathfrak{S}_d$ on $V$ by permuting the $d$ summands. We extend diagonally this $\mathfrak{S}_d$-action to $V^{\otimes n}$. Explicitly, the resulting action of $\sigma\in\mathfrak{S}_d$ permutes the summands of $V^{\otimes n}$ as follows:
\[
  V^{\otimes n}=\!\!\!\bigoplus_{a_1,\dots,a_n=1}^d\!\!\!V_{a_1}\otimes\dots\otimes V_{a_n}\ \ \ \text{and}\ \ \ \sigma\,:\,V_{a_1}\otimes\dots\otimes V_{a_n}\,\to\,V_{\sigma(a_1)}\otimes\dots\otimes V_{\sigma(a_n)}\,.
\]
Together with the action of $A$, this yields an action of $A\rtimes \mathfrak{S}_d$ on $V^{\otimes n}$. We recall from \cref{lem:centralizer_semi} that
\[
  \End_{A\rtimes \mathfrak{S}_d}(V^{\otimes n})= \End_A(V^{\otimes n})^{\mathfrak{S}_d}\ ,
\]
where the action of $\mathfrak{S}_d$ on $\End_A(V^{\otimes n})$ is by conjugating with the action of $\mathfrak{S}_d$ on $V^{\otimes n}$.

\begin{ex}
  Our main example will be $A^{(0)}=U_q(\mathfrak{gl}_N)$ so that 
  \[
    A=U_q(\mathfrak{gl}_N)\boxtimes \dots\boxtimes U_q(\mathfrak{gl}_N)\cong U_q(\mathfrak{gl}_N\oplus \dots\oplus \mathfrak{gl}_N)
  \]
  and the module $V^{(0)}$ is the standard vector representation $\mathbb{C}^N$, so that $V$ is the standard vector representation $\mathbb{C}^{N}\oplus\dots\oplus \mathbb{C}^N$
  by block-diagonal matrices.
\end{ex}

As before, we aim at a description of the centralizer $\End_{A\rtimes \mathfrak{S}_d}(V^{\otimes n})$ as a direct sum of matrix algebras (over tensor products of $A^{(0)}$-centralizers). To do this, we introduce some notations.

Given a $\Bbbk$-algebra $B$ and a finite group $G$, we define the algebra $\mathcal{C}_G(B)$ of $G$-circulant matrices with coefficients in $B$ by
\[
  \mathcal{C}_G(B)=\left\{\bigl(f(h^{-1}g)\bigr)_{g,h\in G}\ |\ f\,:\,G\to B\ \right\}\ .
\]
We see elements of $\mathcal{C}_G(B)$ as matrices of size $|G|$ with coefficients in $B$ by fixing an ordering of the elements of $G$.

\begin{rem}\label{rem:Gcirculant}
  The matrix $\bigl(f(h^{-1}g)\bigr)_{g,h\in G}$ in $\mathcal{C}_G(B)$ is the matrix with coefficients in $B$ representing the right multiplication of the element $\sum_g f(g)g$ on the group ring $B[G]$. Therefore $\mathcal{C}_G(B)$ is an algebra isomorphic to $B[G]\cong B\otimes_{\Bbbk} \Bbbk[G]$. Note that we recover the usual circulant matrices when $G$ is a cyclic group of finite order.
\end{rem}

We are ready to formulate the final result of this section. To lighten the notation, we denote by $B^{(0)}_k$ the endomorphism algebras $\End_{A^{(0)}}((V^{(0)})^{\otimes k})$.

\begin{lem}
  Suppose that for all $r\neq s$ we have $\Hom_{A^{(0)}}((V^{(0)})^{\otimes r},(V^{(0)})^{\otimes s})=0$. Then we have
  \begin{equation}\label{iso:matrix-fixedpoints}
    \End_{A\rtimes \mathfrak{S}_d}(V^{\otimes n}) \simeq \bigoplus_{\substack{\lambda\vdash n \\ l(\lambda)\leq d}} \Mat_{\binom{n}{\lambda}/l_1!\cdots l_n!}\left(\mathcal{C}_{\mathfrak{S}_{l_1}\times\dots\times \mathfrak{S}_{l_n}}\left(B^{(0)}_{\lambda_1}\otimes\cdots\otimes B^{(0)}_{\lambda_d}\right)\right),
  \end{equation}
  where a partition $\lambda$ in the sum is written as $(1^{l_1},2^{l_2},\dots,n^{l_n})$, that is, $l_i$ is the number of $i$ occurring in $\lambda$ and $l(\lambda)=l_1+\cdots+l_n$.
\end{lem}

\begin{proof}
  We introduce some notations. We will denote $\overline{a}=(a_1,\dots,a_n)\in\{1,\dots,d\}^n$. Such an element $\overline{a}$ corresponds to a composition $\nu=(\nu_1,\ldots,\nu_d)\vDash_d n$, where $\nu_i$ counts the number of elements of $\overline{a}$ equal to $i$. We denote $C_{\nu}$ the set of all $\overline{a}$ corresponding to $\nu$. The sets $C_{\nu}$ are the orbits for the $\mathfrak{S}_n$-action on $\{1,\dots,d\}^n$ by place permutations. We can write the space $V^{\otimes n}$ as
  \[
    V^{\otimes n}=\bigoplus_{\nu\vDash_d n}\ \bigoplus_{\overline{a}\in C_{\nu}}V_{a_1}\otimes \cdots \otimes V_{a_n}\ .
  \]
  Now, from \cref{lem:centralizer_semi}, we need to calculate the fixed points under the $\mathfrak{S}_d$-action of $\text{End}_A(V^{\otimes n})$. The $\mathfrak{S}_d$-action comes from the action on $\{1,\dots,d\}^n$ extended diagonally from the natural action on $\{1,\dots,d\}$. This action permutes the sets $C_{\nu}$ corresponding to compositions having the same components in different order. With this $\mathfrak{S}_d$-action, any composition $\nu$ is equivalent to a partition $\lambda$ of $n$ with at most $d$ non zero parts, and we write $\nu\sim_{\mathfrak{S}_d}\lambda$ if this is the case. So now we have
  \[
    V^{\otimes n}=\bigoplus_{\substack{\lambda\vdash n \\ l(\lambda)\leq d}}\bigoplus_{\nu\sim_{\mathfrak{S}_d}\lambda}\bigoplus_{\overline{a}\in C_{\nu}}V_{a_1}\otimes \cdots \otimes V_{a_n}\ .
  \]
The assumption of the lemma implies that there is no $A$-linear morphisms between two subspaces corresponding to different compositions, so we get:
\[
  \End_{A\rtimes \mathfrak{S}_d}(V^{\otimes n})=\bigoplus_{\substack{\lambda\vdash n \\ l(\lambda)\leq d}}\Biggl(\bigoplus_{\nu\sim_{\mathfrak{S}_d}\lambda}\text{End}_A\Bigl(\bigoplus_{\overline{a}\in C_{\nu}}V_{a_1}\otimes \cdots \otimes V_{a_n}\Bigr)\Biggr)^{\mathfrak{S}_d}\ .
\]
For a given $\lambda$, the group $\mathfrak{S}_d$ permutes transitively the summands corresponding to different $\nu$. So a fixed point must have equal components in all these summands, and it is enough to look at the summand corresponding to $\nu=\lambda$. But still there is a residual action inside the summands. Writing $\lambda=(1^{l_1},2^{l_2},\dots,n^{l_n})$, we obtain
\[
  \End_{A\rtimes \mathfrak{S}_d}(V^{\otimes n})\cong\bigoplus_{\substack{\lambda\vdash n \\ l(\lambda)\leq d}}\text{End}_A\Bigl(\bigoplus_{\overline{a}\in C_{\lambda}}V_{a_1}\otimes \cdots \otimes V_{a_n}\Bigr)^{G_{\lambda}}\ ,
\]
where $G_{\lambda}$ is a subgroup of $\mathfrak{S}_d$ isomorphic to $\mathfrak{S}_{l_1}\times\dots\times \mathfrak{S}_{l_n}$. More precisely, $G_{\lambda}$ is the subgroup of elements permuting the numbers in $\{1,\dots,d\}$ which appear with equal multiplicities in $\overline{a}\in C_{\lambda}$. In particular it leaves stable the subset $C_{\lambda}$ of $\{1,\dots,d\}^{n}$. 

Then we decompose the set $C_{\lambda}$ as a union of orbits for the action of $G_{\lambda}$:
\[
  C_{\lambda}=C_{\lambda}^1\sqcup\dots\sqcup C_{\lambda}^k\ .
\]
Note that we have not included in $G_{\lambda}$ the subgroup permuting the numbers in $\{1,\dots,d\}$ appearing with zero multiplicities in $\overline{a}\in C_{\lambda}$. As a result, it is easy to see that the stabilizers of elements of $C_{\lambda}$ for the $G_{\lambda}$-action are all trivial. In other words, each orbit $C_{\lambda}^i$ is equivalent to $G_{\lambda}$ with its left regular action.

Now, as before, we see elements of $\text{End}_A\Bigl(\bigoplus_{\overline{a}\in C_{\lambda}}V_{a_1}\otimes \cdots \otimes V_{a_n}\Bigr)$ as matrices with coefficients in $B^{(0)}_{\lambda_1}\otimes\cdots\otimes B^{(0)}_{\lambda_d}$ and with lines and columns indexed by $C_{\lambda}$. We write such a matrix $M$ as a block matrix using the decomposition of $C_{\lambda}$ into $G_{\lambda}$-orbits. So $M=(M_{ij})_{i,j=1,\dots,k}$ and the block $M_{ij}$ has its lines indexed by $C_{\lambda}^i$ and its columns by $C_{\lambda}^j$. As discussed above, this means that the coefficients of $M_{ij}$ can be indexed by pairs $(g,h)$ where $g,h\in G_{\lambda}$. Finally, it is now immediate to check that the condition of being fixed by $G_{\lambda}$ is equivalent to the condition that for each block $M_{ij}$, the coefficient indexed by $(g,h)$ actually only depends on $h^{-1}g$. This means that each block is an element of $\mathcal{C}_{G_{\lambda}}\bigl(B^{(0)}_{\lambda_1}\otimes\cdots\otimes B^{(0)}_{\lambda_d}\bigr)$.

The cardinal of $C_{\lambda}$ is $\binom{\lambda}{n}$ while the cardinal of each $C_{\lambda}^i$ is $|G_{\lambda}|=l_1!\cdots l_n!$. Thus the number $k$ of blocks in the matrix above is $\binom{\lambda}{n}\frac{1}{l_1!\dots l_n!}$. This concludes the proof.
\end{proof}

\begin{rem}
  One can make a remark similar to \cref{rem-assumptionlemma}. Namely, if we remove the assumptions in the lemma above, we still have a subalgebra of the centralizer isomorphic to the right hand side of \eqref{iso:matrix-fixedpoints}. In fact, what we really have proven is that taking the fixed points under $\mathfrak{S}_d$ of the direct sum of matrix algebras found in \cref{lem:iso-matrix-alg} results in the right hand side of \eqref{iso:matrix-fixedpoints}.
\end{rem}

\subsubsection{Consequences}
\label{sec:consequencesfixedpoints}

As in \cref{sec:consequences}, the isomorphism of \cref{lem:centralizer_semi} implies a Morita equivalence, and in particular a description of the irreducible representations, which can also be obtained by applying Clifford theory to the fixed point subalgebra (see for example \cite[\S 3]{JPA2} and also \cref{rem:Gcirculant}). We have that the irreducible representations are indexed by
\[
  (\lambda,\rho_1,\dots,\rho_d,\Lambda_1,\dots,\Lambda_n)\ \ \ \ \text{where}\
  \left\{
    \begin{array}{l}
      \lambda\vdash n\ \text{with}\ \ell(\lambda)\leq d\,,\\[0.4em]
      \rho_b\in \text{Irr}(B^{(0)}_{\lambda_b})\,,\\[0.4em]
      \Lambda_i \in \text{Irr}_{\Bbbk}(\mathfrak{S}_{l_i})\ ,
    \end{array}
  \right.
\]
where we have denoted a partition $\lambda=(\lambda_1,\dots,\lambda_d)=(1^{l_1}2^{l_2}\dots n^{l_n})$. The dimension of this representation is
\[
  \binom{n}{\lambda}\frac{\prod_{b=1}^d\dim\rho_b\prod_{i=1}^n\dim\Lambda_i}{l_1!\dots l_n!}\ .
\]
Note that the total dimension of the direct sum of matrix algebras in \eqref{iso:matrix-fixedpoints} is
\begin{equation}\label{dim:matrix-fixedpoints}
  \sum_{\substack{\lambda\vdash n \\ l(\lambda)\leq d}}\binom{n}{\lambda}^2\frac{d_{\lambda_1}\dots d_{\lambda_d}}{l_1!\dots l_n!}\,,\ \ \ \ \ \ \text{where $d_k=\dim(B^{(0)}_k)$.} 
\end{equation}
Explicit examples will be given in \cref{sec:tied-braid_centralizers}.
 
\section{Framed braid group and centralizers of tensor products}
\label{sec:framedbraid}

In this section, we keep the setting of the previous section and we will add the assumption that for each algebra $A_b$ and its module $V_b$, we have a morphism from the braid group $\mathcal{B}_n$ to the centralizer $\End_{A_b}(V_b^{\otimes n})$. This assumption will imply that a certain ``framization'' of the braid group naturally appears when looking at the centralizers of the algebra $A$.

In this section, $d$ is still a fixed positive integer and we assume that the field $\Bbbk$ contains a primitive $d$-th root of unity $\zeta$, and that $d$ is invertible in $\Bbbk$. In particular $\Bbbk$ contains $d$ distinct $d$-th roots of unity, which are all powers $\zeta^b$, with $b=0,\dots,d-1$.

\subsection{Braid group and framed braid group}
\label{sec:braid-grp}

\subsubsection{The Artin--Tits braid group of type $A$}
\label{sec:braid-A}

We will denote by $\mathcal{B}_n$ the Artin--Tits braid group of type $A_{n-1}$. A presentation using generators and relations is given by
\[
  \mathcal{B}_n = \left\langle s_1,\ldots,s_{n-1} \middle\vert s_is_j=s_js_i \text{ if } \lvert i-j\rvert > 1, s_is_js_i=s_js_is_j\text{ if } \lvert i-j \rvert = 1\right\rangle.
\]
%

\subsubsection{The framed braid group}
\label{sec:framed-A}

The framed braid group $\mathcal{FB}_{d,n}$ (of type $A$) is defined as the wreath product $\mathcal{B}_n \wr \mathbb{Z}/d\mathbb{Z}$, where the braid group $\mathcal{B}_n$ acts on $(\mathbb{Z}/d\mathbb{Z})^n$ via permutations. An explicit presentation by generators and relations is given by the following
\[
  \mathcal{FB}_{d,n} = 
  \left\langle
    \begin{array}{l}
      s_1,\ldots,s_{n-1},\\
      t_1,\ldots,t_n
    \end{array} 
    \middle\vert
    \begin{array}{ll}
      s_is_j=s_js_i &\text{ if } \lvert i-j\rvert > 1,\\
      s_is_js_i=s_js_is_j & \text{ if } \lvert i-j \rvert = 1,\\
      t_i^d = 1 & \text{ for all } 1 \leq i \leq n,\\
      t_it_j = t_jt_i &\text{ for all } 1 \leq i,j \leq n,\\
      s_it_j = t_{s_i(j)}s_i & \text{ for all } 1 \leq i < n \text{ and } 1 \leq j \leq n
    \end{array}
  \right\rangle.
\]
Here $s_i$ acts on indices $1,\dots,n$ as the transposition $(i,i+1)$.

\subsection{Framed braid group and centralizers}
\label{sec:framed-cent}

For all $b\in\{1,\dots,d\}$, we assume the existence of an element $\check{R}^{(b)}\in \End(V_b\otimes V_b)$ such that, for any $n\geq 2$, we have an algebra morphism given by
\begin{equation}\label{braid_in_Ab}
  \phi_b\ \ \ :\ \
  \begin{array}{rcl} \Bbbk\mathcal{B}_n & \to & \End_{A_b}(V_b^{\otimes n})\\[0.4em]
    s_i & \mapsto & \check{R}^{(b)}_i\ \ \ (i=1,\dots,n-1)\ ,
  \end{array}
\end{equation}
where as usual $\check{R}^{(b)}_i$ denotes the map which acts as $\check{R}^{(b)}$ on factors $i$ and $i+1$ in $V_b^{\otimes n}$. Since $\check{R}^{(b)}_i$ and $\check{R}^{(b)}_j$ obviously commute when $|i-j|>1$, the assumption amounts to the braid relation $\check{R}^{(b)}_i\check{R}^{(b)}_{i+1}\check{R}^{(b)}_i=\check{R}^{(b)}_{i+1}\check{R}^{(b)}_i\check{R}^{(b)}_{i+1}$ being satisfied.

Note that we slightly abuse notations by not indicating the dependence on $n$ for the maps $\phi_b$. The relevant $n$ will always be clear from the context. Note also that $\phi_b$ is also trivially defined for $n=0,1$ since in these cases the braid group $\mathcal{B}_n$ is the trivial group.

\begin{ex}
  If $A_{b}$ is a quasi-triangular Hopf algebra, the action of the $R$-matrix of $A$ on the tensor product $V_{b}\otimes V_{b}$ provides a map $\check{R}^{(b)}$ satisfying the above assumptions.

  Our main examples will be when $A_{b}=U_q(\mathfrak{g})$, namely a quantum enveloping algebra, where $\mathfrak{g}$ is $\mathfrak{gl}_N$ or a complex simple Lie algebra. In this case, $\check{R}^{(b)}$ is obtained through the action of the universal $R$-matrix on a finite-dimensional weight representation $V_{b}$. We refer to \cite{KlSc} for more details on quantum enveloping algebras.
\end{ex}

\begin{rem}
  The braid group does not always generate the centralizer $\End_{A_b}(V_b^{\otimes n})$. If $A_b$ is a quantum enveloping algebra of a simple complex Lie algebra, Lehrer and Zhang \cite{LZ} give a sufficient condition on the representations $V_b$ such that the braid group generates the centralizer algebra.
\end{rem}

Our goal now is to define elements realizing the framed braid group in the centralizer $\End_A(V^{\otimes n})$ of $A$ in $V^{\otimes n}$. First, we define $\tau \colon V \rightarrow V$ by
\[
 \tau(v) = \zeta^{b-1}v\quad\text{ if }v\in V_b,\ \ \forall b=1,\dots,d\ .
\]
With respect to the decomposition $V=\bigoplus_{a=1}^d V_a$, the endomorphism $\tau$ is simply block diagonal, namely, we have:
\[
  \tau=\bigoplus_{b=1}^d\zeta^{b-1}\Id_{V_b}\ .
\]
Then we define $\sigma \colon V\otimes V \rightarrow V\otimes V$ by
\[
  \sigma(v\otimes w) =
  \begin{cases}
    w\otimes v & \text{if } v \in V_b, w\in V_c \text{ with }b\neq c,\\
    \check{R}^{(b)}(v\otimes w) &\text{if } v,w\in V_b.
  \end{cases}
\]
Note that the subspaces $V_b\otimes V_c$ and $V_c\otimes V_b$ are simply permuted by $\sigma$ if $b\neq c$, and $V_b\otimes V_b$ is globally fixed by $\sigma$. An equivalent description of $\sigma$ is:
\[
  \sigma=\bigoplus_{b=1}^d\check{R}^{(b)}|_{V_b\otimes V_b}\oplus \bigoplus_{b\neq c} P|_{V_b\otimes V_c\oplus V_c\otimes V_b}\,,
\]
where $P|_{V_b\otimes V_c\oplus V_c\otimes V_b}$ is the flip operator sending $v\otimes w$ to $w\otimes v$.

To denote endomorphisms of $V^{\otimes n}$, we use the following standard notations:
\[
  \tau_i=\Id_V^{\otimes i-1}\otimes\tau\otimes\Id_V^{\otimes n-i}\ \ \ \text{and}\ \ \ \sigma_i=\Id_V^{\otimes i-1}\otimes\sigma\otimes\Id_V^{n-i-1}\ .
\]
We are ready to state the relation between the framed braid group and the centralizer of $A$ in $V^{\otimes n}$. Recall that the maps $\phi_b$ denote the morphisms from the usual braid group to the centralizers $\End_{A_b}(V_b^{\otimes n})$, set up in \eqref{braid_in_Ab}.

\begin{thm}\label{thm:action-framed}
  Let $n\geq 1$.
  \begin{enumerate}[label=(\roman*)]
  \item We have the following morphism into the centralizer of $A$ in $V^{\otimes n}$:
    \[
      \Phi\ :\
      \begin{array}{rcl}
        \Bbbk\mathcal{FB}_{d,n} & \rightarrow &  \End_A(V^{\otimes n})\\[0.4em]
        t_i & \mapsto & \tau_i\ \ \ \ (i=1,\dots,n)\,,\\[0.4em]
        s_i & \mapsto & \sigma_i\ \ \ \ (i=1,\dots,n-1)\,.
      \end{array}
    \]
  \item We have:
    \begin{equation}\label{iso-imageframed}
      \Phi(\Bbbk\mathcal{FB}_{d,n}) \simeq \bigoplus_{\nu \vDash_d n}\Mat_{\binom{n}{\nu}}\Bigl(\phi_1(\Bbbk\mathcal{B}_{\nu_1})\otimes \cdots \otimes \phi_d(\Bbbk\mathcal{B}_{\nu_d})\Bigr)\,.
    \end{equation}
  \end{enumerate}
\end{thm}

\begin{proof}
  (i) As in the proof of \cref{lem:iso-matrix-alg}, we look at the following decomposition of $V$ as a direct sum of $A$-modules:
  \begin{equation}\label{proof_dec2}
    V^{\otimes n}=\bigoplus_{a_1,\dots,a_n=1}^dV_{a_1}\otimes\dots\otimes V_{a_n}\simeq \bigoplus_{\nu \vDash_d n} \left(V_1^{\otimes \nu_1}\boxtimes \cdots \boxtimes V_d^{\otimes \nu_d}\right)^{\oplus \binom{n}{\nu}}\ ,
  \end{equation}
  where each $(a_1,\dots,a_n)$ corresponds to a composition $\nu$ such that $\nu_j$ is the number of elements among $a_1,\dots,a_d$ which are equal to $j$.
  
  The map $\tau_i$, $i=1,\dots,n$, is equal to $\zeta^{a_i-1}\Id$ on the summand $V_{a_1}\otimes\dots\otimes V_{a_n}$, and therefore commutes with the action of $A$. If $a_i=a_{i+1}=b$, the map $\sigma_i$ is equal to $\check{R}^{(b)}_i$ on $V_{a_1}\otimes\dots\otimes V_{a_n}$, that is it acts as $\check{R}^{(b)}$ on factors $i$ and $i+1$ and as the identity on other factors. This also commutes with $A$, since only $A_{b}$ acts non-trivially on $V_b\otimes V_b$ and its action is centralized by $\check{R}^{(b)}$. Finally, if $a_i\neq a_{i+1}$ then the summand $V_{a_1}\otimes\dots\otimes V_{a_n}$ is isomorphic to the summand with $V_{a_i}$ and $V_{a_{i+1}}$ exchanged. This isomorphism of $A$-modules is given by the flip operator acting on factors $i$ and $i+1$. This coincides with the action of $\sigma_i$ in this case, which therefore commutes with the action of $A$.

  We have shown that the image of the given map indeed takes values in the centralizer of $A$. Now it is straightforward to check that the relations of the framed braid group are satisfied. For example, to check the braid relation between $\sigma_i$ and $\sigma_{i+1}$, the action is non-trivial only on $V_{a_i}\otimes V_{a_{i+1}}\otimes V_{a_{i+2}}$, and one splits the verification in several cases, depending on which among $a_i,a_{i+1},a_{i+2}$ are equal to each other.

  \vskip .1cm
  (ii) First of all, the summands corresponding to different sequences $(a_1,\dots,a_n)$ in \eqref{proof_dec2} are distinguished by the eigenvalues of the commuting operators $\tau_1,\dots,\tau_n$. Therefore, in the image of the framed braid group, we have all projections onto the different summands $V_{a_1}\otimes\dots\otimes V_{a_n}$.

  Then for two summands $V_{a_1}\otimes\dots\otimes V_{a_n}$ and $V_{b_1}\otimes\dots\otimes V_{b_n}$ corresponding to the same composition, one is obtained from the other by a permutation of the indices. The corresponding permutation operator is the corresponding isomorphism of $A$-modules. It is easy to see that such a permutation operator is in the image of the framed braid group, using suitable $\sigma$'s for permuting factors with different indices. 
  
  Now for a composition $\nu$, we consider the simplest summand $V_{a_1}\otimes\dots\otimes V_{a_n}$ corresponding to $\nu$, namely, the summand
  \[
    V_1^{\otimes \nu_1}\otimes \cdots\otimes V_d^{\otimes \nu_d}\ ,
  \]
  which is obtained by taking $a_1=\dots=a_{\nu_1}=1$, $a_{\nu_1+1}=\dots=a_{\nu_1+\nu_2}=2$ and so on. At this point, it remains to show that the subalgebra
  \[
    \phi_1(\Bbbk\mathcal{B}_{\nu_1})\otimes \cdots \otimes \phi_d(\Bbbk\mathcal{B}_{\nu_d})
  \]
  of its endomorphism algebra is obtained in the image of $\Bbbk\mathcal{FB}_{d,n}$. This is the case since this subalgebra is generated by the restrictions of the operators $\sigma_1,\dots,\sigma_{\nu_1-1}$, and $\sigma_{\nu_1+1},\dots,\sigma_{\nu_1+\nu_2-1}$, and so on.
\end{proof}

We can strengthen the conclusion of the preceding theorem if we add some natural assumptions on the centralizers of the algebras $A_b$. First, assume that for all $b=1,\dots,d$, the image of the braid group $\mathcal{B}_n$ generates the centralizer algebra $\End_{A_b}(V_b^{\otimes n})$ for all $n\geq 0$. In other words, the maps $\phi_b$ are all surjective. In this case, the formula in the second item becomes obviously:
\begin{equation}\label{iso-imageframed2}
  \Phi(\Bbbk\mathcal{FB}_{d,n}) \simeq \bigoplus_{\nu \vDash_d n}\Mat_{\binom{n}{\nu}}\Bigl(\End_{A_1}(V_1^{\otimes \nu_1})\otimes \cdots \otimes \End_{A_d}(V_d^{\otimes \nu_d})\Bigr)\,.
\end{equation}
Comparing with \cref{lem:iso-matrix-alg}, we see that this is the full centralizer $\End_{A}(V^{\otimes n})$ if we add the further assumption that $\Hom_{A_b}(V_b^{\otimes r},V_b^{\otimes s}) = 0$ if $r\neq s$. We summarize in the following statement.

\begin{cor}\label{cor:action-framed}
  For all $b=1,\dots,d$, assume that $\Hom_{A_b}(V_b^{\otimes r},V_b^{\otimes s}) = 0$ if $r\neq s$ and that for all $n\geq 0$, the image of the braid group $\mathcal{B}_n$ generates the centralizer algebra $\End_{A_b}(V_b^{\otimes n})$. Then the centralizer algebra $\End_{A}(V^{\otimes n})$ is generated by the image of $\mathcal{FB}_{d,n}$.
\end{cor}

\begin{ex}
  The assumptions of the corollary are satisfied for $U_q(\mathfrak{gl}_N)$ with its fundamental vector representation or for $U_q(\mathfrak{gl}_2)$ with any of its irreducible representations. 
  
  Note that the assumption $\Hom_{A_b}(V_b^{\otimes r},V_b^{\otimes s}) = 0$ if $r\neq s$ will not always be satisfied if $A_b=U_q(\mathfrak{sl}_N)$. For example if $V_b$ is the vector representation, we have $\Hom_{A_b}(V_b^{\otimes {(N+k)}},V_b^{\otimes k}) \neq 0$ for any $k\geq 0$. Here the assumption will be satisfied if $r,s<N$.
\end{ex}

\subsection{The idempotents $E_{i,j}$}\label{sec_Eij}

We define elements $E_{i,j}$, for $i,j=1,\dots,n$, which will be important in the following. They are defined in the commutative subalgebra of the group algebra $\Bbbk\mathcal{FB}_{d,n}$ generated by $t_1,\dots,t_n$:
\[
  E_{i,j}=\frac{1}{d}\sum_{a=1}^dt_i^at_{j}^{-a}\ \ \ \ \ \text{and}\ \ \ \ \ E_i=E_{i,i+1}\ .
\]
The elements $E_{i,j}$ are idempotents, which satisfy obviously $E_{i,j}=E_{j,i}$ and:
\[
  t_iE_{i,j}=t_{j}E_{i,j}\ .
\]
Moreover, their relations with the other generators $s_1,\dots,s_{n-1}$ are
\[
  s_kE_{i,j}=E_{s_k(i),s_k(j)}s_k\ \ \ \ \text{and in particular}\ \ \ s_iE_i=E_is_i\ .
\]
To calculate the images of the idempotents $E_{i,j}$ by $\Phi$, the map given in \cref{thm:action-framed}, we introduce the following operator on $V\otimes V$:
\[
  \varepsilon=\bigoplus_{b=1}^d \Id_{V_b\otimes V_b}\oplus\bigoplus _{b\neq c}0_{V_b\otimes V_c}\ ,
\]
or equivalently
\[
  \varepsilon(v\otimes w) =
  \begin{cases}
    v\otimes w & \text{if } v,w\in V_b\,,\\
    0 & \text{if } v \in V_b, w\in V_c \text{ with }b\neq c\,.
  \end{cases}
\]
For any $i\neq j$, we denote $\varepsilon_{i,j}$ the endomorphism of $V^{\otimes n}$ acting as $\varepsilon$ on factors $i$ and $j$, and as the identity on the other factors. As before, we denote $\varepsilon_i=\varepsilon_{i,i+1}$. 

\begin{prop}\label{action_E}
  We have:
  \[
    \varepsilon_{i,j}=\Phi(E_{i,j})\ \ \ \text{and in particular}\ \ \ \varepsilon_{i}=\Phi(E_i)\ .
  \]
\end{prop}

\begin{proof}
  We have that $\Phi(E_{i,j})$ is $\frac{1}{d}\sum_{a=1}^d\tau_i^a\tau_{j}^{-a}$ and thus if we have a vector in $V_b$ in the $i$-th factor and a vector in $V_c$ in the $j$-th factor, we find that it is multiplied by $\frac{1}{d}\sum_{a=1}^d\zeta^{(b-c)a}$. This is $1$ if $b = c$ and $0$ otherwise.
\end{proof}

\subsubsection{Framization of a characteristic equation} 

The idempotents $E_{i,j}$ and their images $\varepsilon_{i,j}$ are useful to find relations satisfied in the centralizer algebra $\End_A(V^{\otimes n})$ in addition to the relations of the framed braid group. We show how it works when we know a characteristic equation for the maps $\check{R}^{(b)}\in\End(V\otimes V)$, namely we assume that we know non-zero elements $\lambda_1,\dots,\lambda_k$ of $\Bbbk$ such that:
\[
  (\check{R}^{(b)}-\lambda_1)\dots (\check{R}^{(b)}-\lambda_k)=0\,, \ \ \ b=1,\dots,d\,,
\]
The point is that the eigenvalues of $\check{R}^{(b)}$ are the same for all $b=1,\ldots,d$.

This assumption easily implies the following relations for the images of the generators of the framed braid group $\mathcal{FB}_n$ in the centralizer $\End_A(V^{\otimes n})$:
\begin{equation}\label{chareq}
  \varepsilon_i(\sigma_i-\lambda_1)\dots (\sigma_i-\lambda_k)=0\ \ \ \ \text{and}\ \ \ \ (1-\varepsilon_i)(\sigma_i^2-1)=0\,.
\end{equation}
This is immediate, since $\sigma_i$ is the flip operator when restricted to the kernel of the projector $\varepsilon_i$, while it is the direct sum of the operators $\check{R}^{(b)}_i$, $b=1,\dots,d$, when restricted to the kernel of $1-\varepsilon_i$.
We note that any linear combination with two non-zero coefficients of the relations above implies both relations. Explicit examples will be given in \cref{sec:recollection}.

\section{The framed affine braid group and one-boundary centralizers}

In this section, we give the variants of the preceding section involving the affine braid group and its framed version when we consider the one-boundary setting.

\subsection{The affine braid group and its framed version.} \label{sec:braid-B}

We denote by $\mathcal{B}_n^{\mathrm{aff}}$ the affine braid group (or type $B_n$ Artin--Tits braid group) with generators $s_0, s_1, \ldots, s_{n-1}$ and defining relations: 
\begin{equation}\label{affinebraid}
  \begin{array}{rclcl}
    s_0s_1s_0s_1 & = & s_1s_0s_1s_0 && \\[0.1em]
    s_is_j & = & s_js_i && \mbox{for all $i,j=0,1,\ldots,n-1$ such that $\vert i-j\vert > 1$,}\\[0.1em]
    s_is_{i+1}s_i & = & s_{i+1}s_is_{i+1} && \mbox{for  all $i=1,\ldots,n-2$.}\\[0.1em]
  \end{array}
\end{equation}

The framed affine braid group $\mathcal{FB}^{\mathrm{aff}}_{d,n}$ is defined as the wreath product $\mathcal{B}^{\mathrm{aff}}_n \wr \mathbb{Z}/d\mathbb{Z}$, where the braid group $\mathcal{B}^{\mathrm{aff}}_n$ acts on $(\mathbb{Z}/d\mathbb{Z})^n$ as follows. The generators $s_1,\dots,s_{n-1}$ act as permutations as before, and the additional generator $s_0$ acts trivially. An explicit presentation of $\mathcal{FB}^{\mathrm{aff}}_{d,n}$ is with generators $s_0,s_1,\dots,s_{n-1}$ and $t_1, t_2, \ldots, t_n$, with defining relations \eqref{affinebraid} and
\begin{equation}\label{framedaffinebraid}
  \begin{array}{rclcl}
    t_j^d   & =  &  1 && \mbox{for all $j=1,\ldots,n$,}\\[0.1em]
    t_it_j & =  & t_jt_i &&  \mbox{for all $i,j=1,\ldots,n$,}\\[0.1em]
    t_js_0 & = & s_0t_j && \mbox{for all $j=1,\ldots,n$,}\\[0.1em]
    t_js_i & = & s_i t_{s_i(j)} && \mbox{for all  $j=1,\ldots,n$ and $i=1,\ldots,n-1$.}\\[0.1em]
  \end{array}
\end{equation}

\subsection{Framed affine braid group and centralizers} 

As before for the one-boundary situation, we add into the picture an $A$-module $M=M_1\boxtimes \cdots\boxtimes M_d$ where each $M_b$ is an $A_b$-module, and we look at the centralizer of $A$ in $M\otimes V^{\otimes n}$.

Note that we keep our assumption \eqref{braid_in_Ab} of the existence of a morphism from the usual braid group to each centralizer $\End_{A_b}(V_b^{\otimes n})$. Here, we assume moreover that, for all $b\in\{1,\dots,d\}$, we have an element $\check{K}^{(b)}\in \End_{A_b}(M_b\otimes V_b)$ satisfying on $M_b\otimes V_b\otimes V_b$ 
\[
  \check{K}^{(b)}\check{R}^{(b)}\check{K}^{(b)}\check{R}^{(b)}=\check{R}^{(b)}\check{K}^{(b)}\check{R}^{(b)}\check{K}^{(b)}\ ,
\]
where we have extended $\check{K}^{(b)}$ to $M_b\otimes V_b^{\otimes 2}$ by acting trivially on the last factor and $\check{R}^{(b)}$ to $M_b\otimes V_b^{\otimes 2}$ by acting trivially on the first factor. This means that we have the following morphism:
\begin{equation}\label{braidaff_in_Ab}
  \phi^{\mathrm{aff}}_b\ \ \ :\ \
  \begin{array}{rcl} \Bbbk\mathcal{B}^{\mathrm{aff}}_n & \to & \End_{A_b}(M_b\otimes V_b^{\otimes n})\\[0.4em]
    s_0 & \mapsto & \check{K}^{(b)}\ ,\\[0.4em]
    s_i & \mapsto & \check{R}^{(b)}_i\ \ \ (i=1,\dots,n-1)\ ,
  \end{array}
\end{equation}
where we have extended naturally $\check{K}^{(b)}$ to act on $M_b\otimes V_b^{\otimes n}$ by acting trivially on all but the two first factors (also $\check{R}_i^{(b)}$ acts on $V_b^{\otimes n}$ as before and trivially on $M_b$). The morphisms $\phi^{\mathrm{aff}}_b$ extend the morphisms $\phi_b$ from \eqref{braid_in_Ab}.

\begin{ex}
  The main example of such maps arises from the double braiding in a braided category. If $M_b$ and $V_b$ are two objects of a braided category, the map $\check{K}^{(b)}=\check{R}_{V_b,M_b}\check{R}_{M_b,V_b}$ satisfy our assumptions. In the case of quantum groups, this situation have been considered, for example, in \cite{OR}.
\end{ex}

Our goal now is to define elements realizing the framed affine braid group in the centralizer $\End_A(M\otimes V^{\otimes n})$. We have already the action of the framed braid group on $V^{\otimes n}$, realized by the elements $\tau_1,\dots,\tau_n$ and $\sigma_1,\dots,\sigma_{n-1}$. We extend naturally these actions to $M\otimes V^{\otimes n}$ by acting as the identity on $M$. 
So it remains to define an operator $\sigma_0\colon M\otimes V \rightarrow M\otimes V$. Recall that $V=V_1\oplus\dots \oplus V_d$. We will define $\sigma_0$ on each summand in the direct sum
\[
  M\otimes V=M\otimes V_1\oplus\dots\oplus M\otimes V_d\ .
\]
Let $b\in\{1,\dots,d\}$ and write
\[
  M\otimes V_b=(M_1\boxtimes \dots\boxtimes M_b\boxtimes\dots\boxtimes M_d)\otimes V_b\ .
\]
In words, the action of $\sigma_0$ is defined by acting with $\check{K}^{(b)}$ on legs $M_b$ and $V_b$ of the tensor product and trivially on the other factors. More precisely, we define
\[
  \sigma_0\ :\ (m_1\otimes \dots\otimes m_b\otimes\dots\otimes m_d)\otimes v_b\mapsto \sum (m_1\otimes \dots\otimes m'\otimes\dots\otimes m_d)\otimes v'\ ,
\]
where $\check{K}^{(b)}(m_b\otimes v_b)=\sum m'\otimes v'$. Of course $\sigma_0$ is extended to $M\otimes V^{\otimes n}$ by acting trivially on the last $n-1$ factors.

We have the natural one-boundary generalization of \cref{thm:action-framed}. Recall that the maps $\phi^{\mathrm{aff}}_b$ denote the morphisms from the affine braid group to the centralizers $\End_{A_b}(M_b\otimes V_b^{\otimes n})$, set up in \eqref{braidaff_in_Ab}.

\begin{thm}\label{thm:action-framed-aff}
  Let $n\geq 1$.
  \begin{enumerate}[label=(\roman*)]
  \item We have the following morphism into the centralizer of $A$ in $M\otimes V^{\otimes n}$:
    \[
      \Phi^{\mathrm{aff}}\ :\
      \begin{array}{rcl}
        \Bbbk\mathcal{FB}^{\mathrm{aff}}_{d,n} & \rightarrow &  \End_A(M\otimes V^{\otimes n})\\[0.4em]
        t_i & \mapsto & \tau_i\ \ \ \ (i=1,\dots,n)\,,\\[0.4em]
        s_i & \mapsto & \sigma_i\ \ \ \ (i=0,\dots,n-1)\,.
      \end{array}
    \]
  \item We have:
    \begin{equation}\label{iso-imageframed-aff}
      \Phi^{\mathrm{aff}}(\Bbbk\mathcal{FB}^{\mathrm{aff}}_{d,n}) \simeq \bigoplus_{\nu \vDash_d n}\Mat_{\binom{n}{\nu}}\Bigl(\phi^{\mathrm{aff}}_1(\Bbbk\mathcal{B}^{\mathrm{aff}}_{\nu_1})\otimes \cdots \otimes \phi^{\mathrm{aff}}_d(\Bbbk\mathcal{B}^{\mathrm{aff}}_{\nu_d})\Bigr)\,.
    \end{equation}
  \end{enumerate}
\end{thm}

\begin{proof}
  (i) We only need to check that $\sigma_0$ is in the centralizer and that the relations of $\mathcal{FB}^{\mathrm{aff}}_{d,n}$ involving $s_0$ are satisfied. Write the following immediate decomposition as a direct sum of $A$-modules:
  \begin{equation}\label{proof_dec_aff2}
    M\otimes V^{\otimes n}=\bigoplus_{a_1,\dots,a_n=1}^d M\otimes V_{a_1}\otimes\dots\otimes V_{a_n}\ .
  \end{equation}
  On a given summand, the map $\sigma_0$ acts as $\check{K}^{(a_1)}$ on the factors $M_{a_1}$ and $V_{a_1}$ and trivially elsewhere. This commutes with $A$ since only $A_{a_1}$ acts non-trivially on these factors and its action is centralized by $\check{K}^{(a_1)}$.
  
  For the relations of $\mathcal{FB}^{\mathrm{aff}}_{d,n}$ that we need to check, the commutation of $\sigma_0$ with all generators $\tau_i$ is immediate since $\tau_i$ acts as the multiple of the identity $\zeta^{a_i-1}\Id$ on the summand $M\otimes V_{a_1}\otimes\dots\otimes V_{a_n}$, which is stable by $\sigma_0$. The commutation of $\sigma_0$ with $\sigma_2,\dots,\sigma_{n-1}$ is also immediate. For the braid relation involving $s_0$, let us apply it on one given summand $M\otimes V_{a_1}\otimes\dots\otimes V_{a_n}$.
  
  If $a_1=a_2=a$, the relation acts non-trivially only on the factors $M_a,V_{a_1},V_{a_2}$ of the tensor product, and it becomes on these factors the relation $\check{K}^{(a)}\check{R}^{(a)}\check{K}^{(a)}\check{R}^{(a)}=\check{R}^{(a)}\check{K}^{(a)}\check{R}^{(a)}\check{K}^{(a)}$. This is satisfied by our running hypothesis \eqref{braidaff_in_Ab}.
  
  If $a_1=a$ and $a_2=b$ with $a\neq b$, the relation acts non-trivially only on the factors $M_a,M_b,V_{a},V_{b}$ of the tensor product. Ignoring the other factors, let us apply it on a vector of the form $m_a\otimes m_b\otimes v_a\otimes v_b$. Denote $K^{(a)}(m_a\otimes v_a)=\sum m'\otimes v'$ and $K^{(b)}(m_b\otimes v_b)=\sum m''\otimes v''$, and recall that $\sigma_1$ acts on $v_a\otimes v_b$ as the flip. An easy calculation shows that both sides of the relation $\sigma_0\sigma_1\sigma_0\sigma_1=\sigma_1\sigma_0\sigma_1\sigma_0$ gives the same result:
  \[
    m_a\otimes m_b\otimes v_a\otimes v_b\mapsto \sum m'\otimes m''\otimes v'\otimes v''\ .
  \]
  
  \vskip .1cm
  (ii) We can reproduce verbatim the reasoning in the proof of item (ii) of \cref{thm:action-framed} until we are left with checking that we have in the image of $\Phi^{\mathrm{aff}}$ the subalgebra  
  \[
    \phi^{\mathrm{aff}}_1(\Bbbk\mathcal{B}^{\mathrm{aff}}_{\nu_1})\otimes \cdots \otimes \phi^{\mathrm{aff}}_d(\Bbbk\mathcal{B}^{\mathrm{aff}}_{\nu_d})
  \]
  of the endomorphism algebra for the summand
  \[
    M\otimes V_1^{\otimes \nu_1}\otimes \cdots\otimes V_d^{\otimes \nu_d}\ .
  \]
  The first factor $\phi^{\mathrm{aff}}_1(\Bbbk\mathcal{B}^{\mathrm{aff}}_{\nu_1})$ is generated by the operators $\sigma_0,\sigma_1,\dots,\sigma_{\nu_1-1}$. For the action $\phi^{\mathrm{aff}}_2(s_0)$ of the generator $s_0$ of $\mathcal{B}^{\mathrm{aff}}_{\nu_2}$, one needs to consider the operator $\sigma_{\nu_1}\dots \sigma_1\sigma_0\sigma_1^{-1}\dots \sigma_{\nu_1}^{-1}$, since the $\sigma_i$'s around $\sigma_0$ act as permutation operators in this case. The remaining generators of $\phi_2^{\mathrm{aff}}(\Bbbk\mathcal{B}^{\mathrm{aff}}_{\nu_2})$ are obtained with $\sigma_{\nu+1},\dots,\sigma_{\nu_1+\nu_2-1}$. And similarly one can get any $\phi_b^{\mathrm{aff}}(\Bbbk\mathcal{B}^{\mathrm{aff}}_{\nu_b})$ in the subalgebra generated by $\sigma_0,\sigma_1,\dots,\sigma_{n-1}$.
\end{proof}

With the same reasoning as before \cref{cor:action-framed}, we get its analogue for the one-boundary case.

\begin{cor}\label{cor:iso_imageframed_aff}
  Assume that for all $b=1,\dots,d$, the image of the affine braid group $\mathcal{B}^{\mathrm{aff}}_{n}$ generates the centralizer algebra $\End_{A_b}(M_b\otimes V_b^{\otimes n})$ for all $n\geq 0$. Then we have
  \begin{equation}\label{iso-imageframed-aff2}
    \Phi^{\mathrm{aff}}(\Bbbk\mathcal{FB}^{\mathrm{aff}}_{d,n}) \simeq \bigoplus_{\nu \vDash_d n}\Mat_{\binom{n}{\nu}}\Bigl(\End_{A_1}(M_1\otimes V_1^{\otimes \nu_1})\otimes \cdots \otimes \End_{A_d}(M_d\otimes V_d^{\otimes \nu_d})\Bigr)\,.
  \end{equation}
  If moreover we assume that $\Hom_{A_b}(M_b\otimes V_b^{\otimes r},M_b\otimes V_b^{\otimes s}) = 0$ when $r\neq s$, then the centralizer algebra $\End_{A}(M\otimes V^{\otimes n})$ is generated by the image of $\mathcal{FB}^{\mathrm{aff}}_{d,n}$.
\end{cor}

\begin{ex}
  When $A_b=U_q(\mathfrak{gl}_N)$ and $V_b$ is the vector representation, the surjectivity is satisfied \cite{OR}. If $M_b$ is a finite-dimensional irreducible representations, the other assumption is satisfied as well. Again, one has to be careful for this second assumption in other examples such as $U_q(\mathfrak{sl}_N)$.
\end{ex}

\section{Recollection of framizations of algebras}
\label{sec:recollection}

In this section, we consider several algebras appearing in the literature as framizations and we provide a new example for the Birman--Murakami--Wenzl algebra. Each example will be seen as a centralizer algebra for a certain (product of) quantized enveloping algebras.

In this section, we take an indeterminate $q$ and the field $\Bbbk=\mathbb{C}(q)$.

\subsection{The quantum Schur--Weyl duality}
\label{sec:SW}

We recall the well-known statements called quantum Schur--Weyl duality \cite{Jim, Res}. Let $N>1$ and $U_q(\mathfrak{gl}_N)$ denote the quantum group associated to the Lie algebra $\mathfrak{gl}_N$. Let $V$ be the vector representation of $U_q(\mathfrak{gl}_N)$. The centralizer $\End_{U_q(\mathfrak{gl}_N)}(V^{\otimes n})$ is described with the help of the Hecke algebra $\Hec_n$. We fix the normalizations such that the Hecke algebra $\Hec_n$ is defined as the quotient of the braid group algebra $\Bbbk\mathcal{B}_n$ by the following quadratic relations:
\begin{equation}\label{Hecke_relation}
  s_i^2=(q-q^{-1})s_i+1\,,\ \ \ \ \ i=1,\dots,n-1.
\end{equation}

\begin{thm}\label{thm_SW_Hecke}
  There is a surjective morphism from the Hecke algebra to the centralizer:
  \begin{equation}
    \label{Hecke_in_Ab}
    \phi\ \ :\ \ \Hec_n \to \End_{U_q(\mathfrak{gl}_N)}(V^{\otimes n})\ .
  \end{equation}
  It is an isomorphism if and only if $n\leq N$.
\end{thm}

Moreover, if $n>N$, the kernel of $\phi$ is generated by the $q$-antisymmetrizer on $N+1$-points.
For $N=2$, the (unnormalised) $q$-antisymmetrizer on $3$ points is given by
\begin{equation}\label{qAnti_3}
  \Lambda_3=1-q^{-1}s_1-q^{-1}s_2+q^{-2}s_1s_2+q^{-2}s_2s_1-q^{-3}s_1s_2s_1\ .
\end{equation}

\subsection{The Yokonuma--Hecke algebra}
\label{sec:YH}

For each $b=1,\dots,d$, we take $A_b$ to be the quantum group $U_q(\mathfrak{gl}_{N_b})$, for some integer $N_b>1$, and $V_b$ its vector representation of dimension $N_b$. Then the algebra $A$ is 
\[
  A=U_q(\mathfrak{gl}_{N_1})\boxtimes \cdots\boxtimes U_q(\mathfrak{gl}_{N_d})\cong U_q(\mathfrak{gl}_{N_1}\oplus\dots\oplus\mathfrak{gl}_{N_d})\,,
\]
the quantum group associated to the Lie algebra of block-diagonal matrices with block sizes $N_1,\dots,N_d$. The representation $V$ of $A$ is the natural vector representation of dimension $N_1+\cdots+N_d$.

The Yokonuma--Hecke algebra $\YH_{d,n}$ is defined as the quotient of the framed braid group algebra $\Bbbk \mathcal{FB}_{d,n}$ by the additional relation
\begin{equation}\label{YHecke_relation}
  s_i^2=(q-q^{-1})E_{i}s_i+1\,,\ \ \ \ \ i=1,\dots,n-1,
\end{equation}
where we recall that $E_{i}=\frac{1}{d}\sum_{s=1}^dt_i^st_{i+1}^{-s}$.

\begin{thm}\label{thm_YH}
  There exists a surjective homomorphism 
  \[
    \Phi\ \ :\ \ \YH_{d,n}\to\End_{A}(V^{\otimes n})\ .
  \]
  It is an isomorphism if and only if $n \leq N_b$ for all $1 \leq b \leq d$.
\end{thm}

\begin{proof}
  Recall that we already have the algebra morphism
  \[
    \Phi\ :\ \Bbbk\mathcal{FB}_{d,n}\ \rightarrow \ \End_A(V^{\otimes n})
  \]
  following from \cref{thm:action-framed}. From the discussion of the action of $E_i$ in \cref{sec_Eij}, namely Formula \eqref{chareq}, we know that the following relations are satisfied in the image by $\Phi$:
  \[
    E_i(s_i^2-(q-q^{-1})s_i-1)=0\ \ \ \ \text{and}\ \ \ \ (1-E_i)(s_i^2-1)=0\,.
  \]
  The sum of these two relations gives the additional relation \eqref{YHecke_relation} of the algebra $\YH_{d,n}$. Therefore, the morphism $\Phi$ factors through the algebra $\YH_{d,n}$. 
  
  The surjectivity is obtained by an application of \cref{cor:action-framed}, whose hypotheses are satisfied here.
  
  If $n\leq N_b$ for all $1\leq b \leq d$, the centralizers $\End_{A_b}(V_b^{\otimes n})$ are all isomorphic to the Hecke algebra $\Hec_n$, from the usual Schur--Weyl duality in \cref{thm_SW_Hecke}. Thus the second item of \cref{thm:action-framed} reads in this case:
  \begin{equation}\label{iso-YH}
    \Phi(\YH_{d,n}) \simeq \bigoplus_{\nu \vDash_d n}\Mat_{\binom{n}{\nu}}\bigl(\Hec_{\nu_1}\otimes \cdots \otimes \Hec_{\nu_d}\bigr)\,.
  \end{equation}
  The right hand side has dimension
  \[
    \sum_{\nu \vDash_d n}\binom{n}{\nu}^2\nu_1!\dots \nu_d!=n!\sum_{\nu \vDash_d n}\binom{n}{\nu}=n!d^n\ .
  \]
  It is easy to see that $\YH_{d,n}$ is spanned by elements $t_1^{a_1}\dots t_n^{a_n}s_w$, where $a_1,\dots,a_n\in\{1,\dots,d\}$ and elements $s_w$ are indexed by elements $w$ of the symmetric group $\mathfrak{S}_n$ (see for example \cite{CPA,Juy}). This shows that the dimension of $\YH_{d,n}$ is less or equal than $n!d^n$. We conclude that $\Phi$ is an isomorphism.
\end{proof}

\begin{rem}\label{rem:iso_matrix}
  The proof shows that the Yokonuma--Hecke algebra $\YH_{d,n}$ is isomorphic to the direct sum of matrix algebras in the right hand side of \eqref{iso-YH}. This was also shown directly in \cite{JPA} over the ring $\mathbb{C}[q,q^{-1}]$. The representation theory of $\YH_{d,n}$ can be deduced, see \cite{CPA,JPA2} and \cref{sec:consequences}.
\end{rem}

\subsection{Framization of the Temperley--Lieb algebra}
\label{sec:framed-TL}

We keep the setting of the preceding section, $A_b=U_q(\mathfrak{gl}_{N_b})$ and $V_b$ the vector representation of dimension $N_b$, and we consider the case $N_1=\dots=N_d=2$.

The Temperley--Lieb algebra $\TL_n$ is defined as the quotient of the Hecke algebra $\Hec_n$ by the additional relation, if $n>2$,
\[
  1-q^{-1}s_1-q^{-1}s_2+q^{-2}s_1s_2+q^{-2}s_2s_1-q^{-3}s_1s_2s_1=0\ .
\]
Note that this relation implies the same relation with $s_1,s_2$ replaced by $s_i,s_{i+1}$. Since all $N_b$'s are equal to $2$, the Temperley--Lieb algebra is isomorphic to the centralizer $\End_{A_b}(V_b^{\otimes n})$ for any $b=1,\dots,d$, as was recalled in \cref{sec:SW}.

The following analogue of the Temperley--Lieb algebra in the framized situation was defined in \cite{GJKL}, see also \cite{CP,Gou}.

\begin{defn}
  The framization of the Temperley--Lieb algebra, denoted $\FTL_{d,n}$, is the quotient of the Yokonuma--Hecke algebra $\YH_{d,n}$ by the relation:
  \begin{equation}\label{rel-FTL}
    E_1E_2(1-q^{-1}s_1-q^{-1}s_2+q^{-2}s_1s_2+q^{-2}s_2s_1-q^{-3}s_1s_2s_1)=0\ .
  \end{equation}
\end{defn}

Here again, the same relation with indices $1,2$ replaced by $i,i+1$ is implied. Note that from the properties of the elements $E_i$ recalled in \cref{sec_Eij}, the product $E_1E_2$ commutes with $s_1$ and $s_2$.

\begin{thm}\label{thm_FTL}
  The algebra $\FTL_{d,n}$ is isomorphic to the centralizer $\End_{U_q(\mathfrak{gl}_2^d)}(V^{\otimes n})$ where $V$ is the vector representation of dimension $2d$.
\end{thm}

\begin{proof}
  We already have the surjective morphism
  \[
    \Phi\ \ :\ \ \YH_{d,n}\ \to\ \End_{U_q(\mathfrak{gl}_2^d)}(V^{\otimes n})
  \]
  from \cref{thm_YH}. First we need to check that the defining relation \eqref{rel-FTL} of $\FTL_{d,n}$ is satisfied in the image by $\Phi$. From the description of the image by $\Phi$ of the idempotents $E_i$ in \cref{action_E}, we have that the product $E_1E_2$ acts as follows on $V\otimes V\otimes V$: it acts as the identity on subspaces of the form $V_a\otimes V_a\otimes V_a$, where $a=1,\dots,d$, and acts as $0$ on all other subspaces $V_a\otimes V_b\otimes V_c$. 

  On subspaces $V_a\otimes V_a\otimes V_a$ the relation \eqref{rel-FTL} is satisfied since the $q$-antisymmetrizer $\Lambda_3$ acts as $0$ from the usual Schur--Weyl duality, and on other subspaces it is trivially satisfied since $E_1E_2$ acts as $0$.

  Therefore, the morphism $\Phi$ factors through the algebra $\FTL_{d,n}$. The second item of \cref{thm:action-framed} reads in this case:
  \begin{equation}\label{iso-FTL}
    \Phi(\FTL_{d,n}) \simeq \bigoplus_{\nu \vDash_d n}\Mat_{\binom{n}{\nu}}\bigl(\TL_{\nu_1}\otimes \cdots \otimes \TL_{\nu_d}\bigr)\,.
  \end{equation}
  To conclude that $\Phi$ is an isomorphism, one can show that the dimension of $\FTL_{d,n}$ is less or equal than the dimension of the right hand side. This can be found in \cite{CP}.
\end{proof}

In \cite{CP} the algebra $\FTL_{d,n}$ is directly shown to be isomorphic to the direct sum in \eqref{iso-FTL} and the representation theory is described. This is an example of the general setting in \cref{sec:consequences}.

\subsubsection{The Complex Temperley--Lieb algebra.} 

Variants of the algebra $\FTL_{d,n}$ were defined, see \cite{CP,Gou}. In particular, the so-called complex Temperley--Lieb algebra $\CTL_{d,n}$ is defined as the quotient of the Yokonuma--Hecke algebra $\YH_{d,n}$ by the relation:
\begin{equation}\label{rel-CTL}
  \frac{1}{d^3}\sum_{a,b,c=1}^dt_1^at_2^bt_3^c(1-q^{-1}s_1-q^{-1}s_2+q^{-2}s_1s_2+q^{-2}s_2s_1-q^{-3}s_1s_2s_1)=0\ .
\end{equation}
From the point of view of the action of the algebra $\YH_{d,n}$ on $V^{\otimes n}$ from the preceding sections, the meaning of the prefactor 
\[
  \sum_{a,b,c=1}^dt_1^at_2^bt_3^c=(\frac{1}{d}\sum_{a=1}^dt_1^a)E_1E_2
\]
is easy to explain. It acts on $V\otimes V\otimes V$ as follows: it is proportional to the identity on the subspace $V_1\otimes V_1\otimes V_1$, and it acts as $0$ on all other subspaces $V_a\otimes V_b\otimes V_c$. In particular, on all these latter subspaces, the relation \eqref{rel-CTL} is automatically satisfied. On the subspace $V_1\otimes V_1\otimes V_1$ it is satisfied only if the dimension of $V_1$ is $2$. 

From the preceding discussion, using the same reasoning as in the preceding subsection, we get the following interpretation of the algebra $\CTL_{d,n}$ as a centralizer. We need to take $A=U_q(\mathfrak{gl}_2\oplus \mathfrak{gl}_{N_2}\oplus \dots\oplus \mathfrak{gl}_{N_d})$ (and $V$ is of dimension $2+N_2+\dots+N_d$). 

\begin{thm}\label{thm_CTL}
  There exists a surjective homomorphism 
  \[
    \Phi\ \ :\ \ \CTL_{d,n}\to\End_{A}(V^{\otimes n})\ .
  \]
  It is an isomorphism if and only if $n \leq N_b$ for all $2 \leq b \leq d$.
\end{thm}

Again, as in the preceding subsection, we recover the isomorphism
\[
  \CTL_{d,n} \simeq \bigoplus_{\nu \vDash_d n}\Mat_{\binom{n}{\nu}}\bigl(\TL_{\nu_1}\otimes \Hec_{\nu_2}\otimes \cdots \otimes \Hec_{\nu_d}\bigr)\,,
\]
which was proved directly in \cite{CP} along with the representation theory of $\CTL_{d,n}$, which is a particular case of \cref{sec:consequences}.

\begin{rem}
  It is straightforward to generalize the above picture, by taking various idempotents and various $q$-antisymmetrizers in relations similar to \eqref{rel-FTL} or \eqref{rel-CTL}, in order to relate to centralizers $\End_{A}(V^{\otimes n})$ for various values of $N_1,\dots,N_d$. For example, one can take $s\in\{1,\dots,d\}$ and replace the prefactor in \eqref{rel-CTL} by $\sum_{a,b,c=1}^d\zeta^{-a(s-1)}t_1^at_2^bt_3^c$ to relate to the centralizers when $N_s$ is of dimension 2 and other $N_b$'s arbitrary.
\end{rem}

\begin{rem}\label{rem-TL}
  We conclude from this subsection, as was also advocated in \cite{CP}, that the most natural framized versions of the Temperley--Lieb algebra are, first, the algebra $\FTL_{d,n}$ and, second, the algebra $\CTL_{d,n}$. The other variant called Yokonuma--Temperley--Lieb algebra, see \cite{Gou}, does not seem to be naturally related to any centralizer.
\end{rem}

\subsection{Framization of the Birman--Murakami--Wenzl algebra}
\label{sec:framed-BMW}

In this section, we work over the field $\Bbbk=\mathbb{C}(q,a)$ with two indeterminates. The Birman--Murakami--Wenzl algebra, BMW algebra for short, is defined as the quotient of the braid group algebra $\Bbbk \mathcal{B}_n$ by the additional relations
\begin{align}
  e_i s_i & =  a^{-1} e_i & & \mbox{for} \ i=1,\dots, n-1\,, \label{BMW1}\\
  e_i s_{j}^{\pm1} e_i & = a^{\pm1} e_{i} & & \mbox{for} \ |i-j|=1\,,   \label{BMW2}
\end{align}
where
\begin{equation}
  \label{defei}
  e_i = 1- \frac{s_i-s_i^{-1}}{q - q^{-1}} \ .
\end{equation}
We will denote this algebra by $\BMW_n(q,a)$ or $\BMW_n$ if the parameters are clear from the context. As a consequence of the defining relations, the generators $s_i$'s satisfy a cubic relation:
\begin{equation}
  (s_i-a^{-1})(s_i^2-(q-q^{-1})s_i-1)=0. \label{BMW6}
\end{equation}
The quotient by the relations $e_i=0$ gives back the Hecke algebra. Other implied relations in $\BMW_n(q,a)$ are
\begin{align}
  e_i^2 & = \left( \frac{a - a^{-1}}{q - q^{-1}} +1 \right) e_{i}\,, \label{BMW5}\\
  e_i e_{j} e_i & = e_{i} & \mbox{for} \ |i-j|=1\,. \label{BMW3}
\end{align}
The algebra $\BMW_n$ can be seen as a deformation of the Brauer algebra and its dimension is equal to $ \frac{(2n)!}{2^n n!} = (2n-1) \cdot (2n-3) \cdots 5 \cdot 3 \cdot 1$.

An instance of Schur--Weyl duality \cite[Sections 5 and 6]{Res} shows that for specific specializations of $a$, this algebra is related to a centralizer algebra for $U_q(\mathfrak{sp}_{2N})$ or $U_q(\mathfrak{so}_{2N})$. Let $V$ be the vector representation of $U_q(\mathfrak{sp}_{2N})$ or $U_q(\mathfrak{so}_{2N})$.

\begin{thm}~
  \label{thm:SW_BMW}
  \begin{enumerate}[label=(\roman*)]
  \item Specialize $a$ to $q^{N-1}$. There is a surjective morphism from the BMW algebra to the centralizer
    \begin{equation}
      \label{eq:SW_BMW_so}
      \phi\ \ :\ \ \BMW_{n}(q,q^{N-1})\to\End_{U_q(\mathfrak{so}_{2N})}(V^{\otimes n})\ .
    \end{equation}
    It is an isomorphism if and only if $n\leq N$.
  \item Specialize $a$ to $-q^{N+1}$. There is a surjective morphism from the BMW algebra to the centralizer
    \begin{equation}
      \label{eq:SW_BMW_sp}
      \phi\ \ :\ \ \BMW_{n}(q,-q^{N+1})\to\End_{U_q(\mathfrak{sp}_{2N})}(V^{\otimes n})\ .
    \end{equation}
    It is an isomorphism if and only if $n\leq N$.  
  \end{enumerate}
\end{thm}

\begin{rem}
  A similar statement also exists for $U_q(\mathfrak{so}_{2N+1})$ but one needs to add a square root of $q$.
\end{rem}

In \cite{JuLa}, a definition of the framization of the BMW algebra is proposed and seems not to be related to the context of the present article. We propose a different definition for the framization of the BMW algebra that we can relate with a centralizer. Recall the idempotents $E_i$ introduced in \cref{sec_Eij}.

\begin{defn}
  \label{def:FBMW}
  The framization of the BMW algebra, denoted $\FBMW_{d,n}(q,a)$, is the quotient of the framed braid group algebra $\Bbbk\mathcal{FB}_{d,n}$ by the additional relations
  \begin{align}
    e_i s_i & =  a^{-1} e_i & & \mbox{for} \ i=1,\dots, n-1\,, \label{FBMW1}\\
    e_i s_{i+1}^{\pm1} e_i E_{i+1} & = a^{\pm1} e_{i}E_{i+1} & & \mbox{for} \ i=1,\dots, n-2\,,   \label{FBMW2}
  \end{align}
  where
  \begin{equation}
    \label{defei-framed}
    e_i = E_i - \frac{s_i-s_i^{-1}}{q - q^{-1}}\ .
  \end{equation}
\end{defn}
Note that the definition of $e_i$ in the algebra $\FBMW_{d,n}(q,a)$ involves the idempotent $E_i$. As a consequence, the cubic relation \eqref{BMW6} is replaced by
\begin{equation}
  (s_i-a^{-1})(s_i^2-(q-q^{-1})s_iE_i-1)=0. \label{FBMW6}
\end{equation}
The quotient by the relations $e_i=0$ now gives back the Yokonuma--Hecke algebra. As for the usual BMW algebra, some additional relations are implied.
\begin{lem}\label{lem:FBMW}
  The following relations are satisfied in $\FBMW_{d,n}(q,a)$:
  \begin{align}
    e_iE_i & = e_i, \label{FBMWeE}\\
    e_i^2 &  = \left( \frac{a - a^{-1}}{q - q^{-1}} +1 \right) e_{i}\,, \label{FBMW5}\\
    e_iE_j & =E_je_i & \mbox{for all}\ i,j\,,\label{FBMWeE2}\\
    e_i e_{j} e_i & = e_{i} E_{j} & \mbox{for}\ |i-j|=1\ .  \label{FBMW3}
  \end{align}
\end{lem}

\begin{proof}
  Using the definition of $e_i$, it is clear that $e_it_i = t_{i+1}e_i$. Therefore
  \[
    t_ie_i = at_is_ie_i = as_it_{i+1}e_i = as_ie_it_i = e_it_i = t_{i+1}e_i
  \]
  and $t_it_{i+1}^{-1}e_i = e_i$. Then Relation \eqref{FBMWeE} follows from the definition of $E_i$ in terms of $t_i,t_{i+1}$.
  
  Relation \eqref{FBMW5} follows immediately from \eqref{FBMWeE}. 
  
  As for Relation \eqref{FBMWeE2}, first note that $s_i$ commutes with $E_j$ if $j\neq i+1$. Moreover, $s_i$ commutes with $E_{i+1}E_i$. Recall also that all $E_i$'s commute. Now we claim that these facts together with \eqref{FBMWeE} imply Relation \eqref{FBMWeE2}. First, for $j\neq i+1$, this is immediate. Second, for $j=i+1$, we have:
  \[
    E_{i+1}e_i=E_{i+1}E_ie_i=e_iE_{i+1}E_i=e_iE_iE_{i+1}=e_iE_{i+1}\ .
  \]
  Finally, Relation \eqref{FBMW3} immediately follows replacing $e_j$ by its definition and using the previous relations.
\end{proof}

The main purpose of \cref{def:FBMW} is that the framization of the BMW algebra relates to some centralizers of $U_q(\mathfrak{so}_{2N}\oplus\dots\oplus \mathfrak{so}_{2N})$ and $U_q(\mathfrak{sp}_{2N}\oplus\dots\oplus \mathfrak{sp}_{2N})$. Note that, unlike the $\mathfrak{gl}_N$ situation with the Yokonuma--Hecke algebra, here we have a single integer $N$ involved. This is because this dimension fixes the value of the parameter $a$ and therefore can not vary.

\begin{prop}~\label{prop:FBMW}
  \begin{enumerate}[label=(\roman*)]
  \item For each $b=1,\ldots,d$, we choose the algebra $A_b = U_q(\mathfrak{so}_{2N})$ and $V_{b}$ its vector representation of dimension $2N$. There exists a homomorphism
    \[
      \Phi\ \ :\ \ \FBMW_{d,n}(q,q^{N-1})\to\End_{A}(V^{\otimes n})\ .
    \]
  \item For each $b=1,\ldots,d$, we choose the algebra $A_b = U_q(\mathfrak{sp}_{2N})$ and $V_{b}$ its vector representation of dimension $2N$. There exists a homomorphism
    \[
      \Phi\ \ :\ \ \FBMW_{d,n}(q,-q^{N+1})\to\End_{A}(V^{\otimes n})\ .
    \]
  \end{enumerate}
\end{prop}

\begin{proof}
  We only prove the case of $A=U_q(\mathfrak{so}_{2N})^{\otimes d}$, the case of $\mathfrak{sp}_{2N}$ is similar. We already have the algebra morphism
  \[
    \Phi\ \ :\ \ \Bbbk\mathcal{FB}_{d,n}\to\End_{A}(V^{\otimes n})
  \]
  following from \cref{thm:action-framed-aff}. It suffices to show that the relations of the framed BMW algebra are satisfied. Let $v=v_1\otimes \cdots \otimes v_n \in V^{\otimes n}$ with $v_i$ being in the summand $V_{a_i}$ of $V$. 
  
  Let us start with \eqref{FBMW1}. If $a_i\neq a_{i+1}$, then $E_i$ acts by $0$ and $s_i$ acts on $v$ by permuting the components $v_{a_i}$ and $v_{a_{i+1}}$. Therefore the element $e_i$ acts by $0$ on $v$ and the relation \eqref{FBMW1} is satisfied. If $a_{i}=a_{i+1}$, then $E_i$ acts by the identity on $v$ and the relation \eqref{FBMW1} is satisfied by \cref{thm:SW_BMW}.
  
  Relation \eqref{FBMW2} is proven in the same fashion. The term $e_iE_{i+1}$ acts on $v$ by $0$ unless $a_i=a_{i+1}=a_{i+2}$ where it acts by the identity on $v$. In that second case, Relation \eqref{FBMW2} is once again a consequence of \cref{thm:SW_BMW}.
\end{proof}

From \cref{sec:framedbraid}, since the morphisms $\phi_b$ are surjective, we have in both cases  above:
\begin{equation}\label{iso-imageframed2'}
  \Phi(\FBMW_{d,n}) \simeq \bigoplus_{\nu \vDash_d n}\Mat_{\binom{n}{\nu}}\Bigl(\End_{A_1}(V_1^{\otimes \nu_1})\otimes \cdots \otimes \End_{A_d}(V_d^{\otimes \nu_d})\Bigr)\,.
\end{equation}
where for simplicity we omit the parameters of the algebras. Moreover, if $n\leq N$, the morphisms $\phi_b$ are isomorphisms, and therefore
\begin{equation}\label{iso-imageframed3'}
  \Phi(\FBMW_{d,n}) \simeq \bigoplus_{\nu \vDash_d n}\Mat_{\binom{n}{\nu}}\Bigl(\BMW_{\nu_1}\otimes \cdots \otimes \BMW_{\nu_d}\Bigr)\,.
\end{equation}
At this point it is natural to ask about the injectivity of $\Phi$. One way to answer this is to have an upper bound on the dimension of $\FBMW_{d,n}$ corresponding to the dimension of the right hand side of \eqref{iso-imageframed3'}. 

Actually, we conjecture the following natural isomorphism theorem similar to the ones obtained for the Yokonuma--Hecke algebra and their Temperley--Lieb versions. 

\begin{conj}\label{conjFBMW}
  Over some subring of $\mathbb{C}(q,a)$, we have
  \begin{equation}\label{iso-BMW}
    \FBMW_{d,n} \simeq \bigoplus_{\nu \vDash_d n}\Mat_{\binom{n}{\nu}}\Bigl(\BMW_{\nu_1}\otimes \cdots \otimes \BMW_{\nu_d}\Bigr)\,.
  \end{equation}
\end{conj}

In particular, we conjecture that the dimension of the algebra $\FBMW_{d,n}$ is:
\[
  \dim\FBMW_{d,n}=\sum_{\nu\vDash_d n}\binom{n}{\nu}^2\frac{(2\nu_1)!\dots(2\nu_d)!}{2^n\nu_1!\dots \nu_d!}\ .
\]
This was checked for small values of $n$ and $d$. For $d=2$, the sequence of dimensions start with $1,2,10,84,1014,16140$ for $n=0,1,2,3,4,5$. This does not seem to be on \cite{OEIS}.

\subsection{Affine and cyclotomic Yokonuma--Hecke algebras}
\label{sec:cyclotomic-YH}

For $b=1,\dots,d$, we take once again $A_b=U_q(\mathfrak{gl}_{N_b})$ and $V_b$ the vector representation of dimension $N_b$. We also take a module $M_b$ in the category $\mathcal{O}$ for $U_q(\mathfrak{gl}_{N_b})$. In this case, we have for each $b$ a morphism
\[
  \phi^{\mathrm{aff}}_b\ :\ \Hec_n^{\mathrm{aff}}\to\End_{A_b}(M_b\otimes V_b^{\otimes n})\ ,
\]
where the affine Hecke algebra $\Hec_n^{\mathrm{aff}}$ is the quotient of the algebra of the affine braid group $\Bbbk[\mathcal{B}_n^{\mathrm{aff}}]$ by the Hecke relation $s_i^2=(q-q^{-1})s_i+1$ for all $i=1,\dots,n-1$. We refer to \cite{OR} where it is also shown that this morphism is surjective if $M_b$ is a finite-dimensional irreducible module. Another situation where $\phi^{\mathrm{aff}}_b$ is surjective is when $M_b$ is a parabolic universal Verma module, see \cite{LV}.

Now recall from \cite{CPA} the definition of the affine Yokonuma--Hecke algebra $\YH_{d,n}^{\mathrm{aff}}$ as the quotient of the algebra $\Bbbk[\mathcal{FB}_{d,n}^{\mathrm{aff}}]$ of the framed affine braid group by the quadratic relation $s_i^2=(q-q^{-1})E_is_i+1$ for all $i=1,\dots,n-1$. The following result is obtained immediately combining the results from \cref{sec:braid-B} with the calculation already made in \cref{thm_YH}.

\begin{thm}
  \label{thm_YHaff}
  There exists a homomorphism 
  \[
    \Phi^{\mathrm{aff}}\ \ :\ \ \YH^{\mathrm{aff}}_{d,n}\to\End_{A}(M\otimes V^{\otimes n})\ .
  \]
\end{thm}

It is easily obtained using the results from \cref{sec:braid-B} that it is surjective if for example each $M_b$ is a finite-dimensional irreducible module or a universal parabolic Verma module.

Now for simplicity, assume that all $N_b$'s are equal to a number $N$ and that all $M_b$'s are the same irreducible $U_q(\mathfrak{gl}_N)$-module $M^{(0)}$. Then all maps $\phi_b^{\mathrm{aff}}$ factors through the same cyclotomic quotient (or Ariki--Koike algebra)
\[
  \phi^{\mathrm{aff}}_b\ :\ \Hec_{m,n}^{\mathrm{cyc}}\to\End_{A_b}(M_b\otimes V_b^{\otimes n})\ ,
\]
where $\Hec_{m,n}^{\mathrm{cyc}}$ is the quotient of $\Hec_n^{\mathrm{aff}}$ by the relation $(s_0-\lambda_1)\dots (s_0-\lambda_m)=0$, where the eigenvalues $\lambda_1,\dots,\lambda_m$ depends on the choice of the module $M^{(0)}$. The number $m$ of eigenvalues is the number of irreducible components in the decomposition of $M^{(0)}\otimes V^{(0)}$ and these eigenvalues are computed in \cite[Remark after Proposition 5.1]{DR}.

In this situation, we have obviously that the morphism $\Phi^{\mathrm{aff}}$ from \cref{thm_YHaff} factors through a quotient of $\YH^{\mathrm{aff}}_{d,n}$ called in \cite{CPA} the cyclotomic Yokonuma--Hecke algebra. This is defined as the quotient of $\YH^{\mathrm{aff}}_{d,n}$ by the same relation $(s_0-\lambda_1)\dots (s_0-\lambda_m)=0$ and we denote this algebra by $\YH^{\mathrm{cyc}}_{d,m,n}$.

We now give conditions when the map $\Phi^{\mathrm{aff}}$ from Theorem \eqref{thm_YHaff} is an isomorphism between the cyclotomic Yokonuma--Hecke algebra and the endomorphism algebra $\End_{A}(M\otimes V^{\otimes n})$. Suppose that $M^{(0)}$ is the finite dimensional representation of $U_q(\mathfrak{gl}_N)$ associated to the partition $\mu$, which is of length at most $N$. The number of summands of $M^{(0)}\otimes V^{(0)}$ is given by the number of partitions of length at most $N$ obtained from $\mu$ by adding one box. Denote by $1=r_1<r_2<\cdots<r_m\leq N$ (resp. $c_1>c_2>\cdots>c_m$) the rows (resp. columns) of addable boxes of $\mu$. 

\begin{lem}\label{lem:iso_YH_glN}
  The map $\phi^{\mathrm{aff}}_b \colon \Hec_{m,n}^{\mathrm{cyc}}\rightarrow \End_{A^{(0)}}(M^{(0)}\otimes (V^{(0)})^{\otimes n})$ is an isomorphism if and only if $n\leq c_i-c_{i+1}$, $n\leq r_{i+1}-r_i$ and $n\leq N+1-r_m$ for all $1 \leq i < m$.
\end{lem}

\begin{proof}
  This follows from \cite[Theorem 6.20]{OR}, which provides the dimension of the endomorphism algebra. An equivalent argument would be to compare the Bratelli diagram describing the branching rule of $M^{(0)}\otimes (V^{(0)})^{\otimes n}$ and the poset of $m$-partitions: they are equal at the levels $n$ satisfying the condition of the lemma, which implies the equality of dimensions of $\Hec_{m,n}^{\mathrm{cyc}}$ and $\End_{A^{(0)}}(M^{(0)}\otimes (V^{(0)})^{\otimes n})$.
\end{proof}

Using the results obtained in \cref{sec:framedbraid} we immediately obtain the following proposition.

\begin{prop}\label{prop:iso_cyclo_YH_glN}
  The map $\Phi^{\mathrm{aff}}\colon \YH^{\mathrm{cyc}}_{d,m,n} \rightarrow \End_{A^{(0)}}(M\otimes V^{\otimes n})$ is an isomorphism if and only if $n\leq c_i-c_{i+1}$, $n\leq r_{i+1}-r_i$ and $n\leq N+1-r_m$ for all $1 \leq i < m$.
\end{prop}

\begin{proof}
  By \cref{cor:iso_imageframed_aff} we have an isomorphism
  \[
    \Phi^{\mathrm{aff}}(\YH^{\mathrm{cyc}}_{d,m,n}) \simeq \bigoplus_{\nu \vDash_d n}\Mat_{\binom{n}{\nu}}\Bigl(\End_{A^{(0)}}(M^{(0)}\otimes (V^{(0)})^{\otimes \nu_1})\otimes \cdots \otimes \End_{A^{(0)}}(M^{(0)}\otimes (V^{(0)})^{\otimes \nu_d})\Bigr)\,.
  \]

  Using \cref{lem:iso_YH_glN}, we obtain that $\End_{A^{(0)}}(M^{(0)}\otimes (V^{(0)})^{\otimes \nu_k}) \simeq H_{m,n}^{\mathrm{cyc}}$ for all $1\leq k \leq d$. Therefore $\Phi^{\mathrm{aff}}(\YH^{\mathrm{cyc}}_{d,m,n})$ is of dimension
  \[
    \sum_{\nu \vDash_d n}\binom{n}{\nu}^2m^{\nu_1}\nu_1!\cdots m^{\nu_d}\nu_d! = m^nn!\sum_{\nu \vDash_d n}\binom{n}{\nu} = (dm)^nn! = \dim(\YH^{\mathrm{cyc}}_{d,m,n}),
  \]
  and therefore $\Phi^{\mathrm{aff}}$ is an isomorphism. The last equality above is proved in \cite{CPA2}.
\end{proof}

As in the previous sections, we recover an isomorphism proved in \cite{Pou}
\[
  \YH^{\mathrm{cyc}}_{d,m,n} \simeq \bigoplus_{\nu\vDash_d n}\Mat_{\binom{n}{\nu}}(\Hec^{\mathrm{cyc}}_{m,\nu_1}\otimes \cdots \otimes \Hec^{\mathrm{cyc}}_{m,\nu_d}).
\]
Indeed, given $d,n$ and $m$, it suffices to find a partition $\mu$ and an integer $N$ large enough such that the conditions of \cref{prop:iso_cyclo_YH_glN} are satisfied. For example, one can choose $N\geq mn$ and $\mu=(((m-1)n)^n,((m-2)n)^n,\ldots,n^n)$, the exponent being repetition of entries:
\[
  \mu =
  \begin{tikzpicture}[baseline={(0,-1.3)},scale=.6]
    \draw[thick] (2,-3) -- (1,-3) -- (1,-4) -- (0,-4) -- (0,0) -- (4,0) -- (4,-1) -- (3,-1) -- (3,-2);
    \draw[thick, decoration={brace, raise=.1em},decorate] (4,0) -- (4,-1) node[anchor=south,yshift=.1em,xshift=.7em] {$n$};
    \draw[thick, decoration={brace, raise=.1em},decorate] (3,-1) -- (3,-2) node[anchor=south,yshift=.1em,xshift=.7em] {$n$};
    \draw[thick, decoration={brace, raise=.1em},decorate] (1,-3) -- (1,-4) node[anchor=south,yshift=.1em,xshift=.7em] {$n$};
    \draw[thick, decoration={brace, mirror, raise=.1em},decorate] (4,-1) -- (3,-1) node[anchor=south,yshift=.1em,xshift=.75em] {$n$};
    \draw[thick, decoration={brace, mirror, raise=.1em},decorate] (2,-3) -- (1,-3) node[anchor=south,yshift=.1em,xshift=.75em] {$n$};
    \draw[thick, decoration={brace, mirror, raise=.1em},decorate] (1,-4) -- (0,-4) node[anchor=south,yshift=.1em,xshift=.75em] {$n$};
    \node at (2.5,-2.3) {$\iddots$};
  \end{tikzpicture}
\]

Similar results can be obtained if $M^{(0)}$ is a universal parabolic Verma module, using the surjectivity result in \cite[Theorem 4.2]{LV}.

\begin{rem}
  We could as well consider specific choices of $N$ and of $M^{(0)}$, where the quotient of $\Hec_{m,n}^{\mathrm{cyc}}$ isomorphic to the endomorphism algebra of $M^{(0)}\otimes (V^{(0)})^{\otimes n}$ has an explicit description in terms of generators and relations. For example,
  \begin{enumerate}
  \item if $N=2$ and $M^{(0)}$ is the $k$-th symmetric power of the vector representation of $\mathfrak{gl}_2$, then for $n\leq k$, the centralizer of $M^{(0)}\otimes (V^{(0)})^{\otimes n}$ is isomorphic to a specialization of the blob algebra of Martin and Saleur \cite{MS}, also known as the one-boundary Temperley--Lieb algebra. See also \cite{PZ} for generalizations to $\mathfrak{gl}_N$.
  \item if $M^{(0)}$ is the irreducible representation of $\mathfrak{gl}_N$ given by the partition $((N-1)k,(N-2)k,\ldots,k)$, then for $n\leq k$, the centralizer of $M^{(0)}\otimes (V^{(0)})^{\otimes n}$ is isomorphic to a specialization of the generalized blob algebra of Martin and Woodcock \cite{MW}.
  \end{enumerate}

  In these two cases, we leave to the reader to find a presentation by generators and relations of the quotient of $\YH^{\mathrm{cyc}}_{d,m,n}$ ($m=2$ in the first case, $m=N$ in the second case) isomorphic to the centralizer $\End_A(M\otimes V^{\otimes n})$, thereby defining framizations of the one-boundary Temperley--Lieb algebra and of the generalized blob algebra.
\end{rem}

\section{Tied braid algebra and fixed points subalgebra of $\Bbbk[\mathcal{FB}_{d,n}]$}
\label{sec:tied-braid_fixedpoints}

In the section, we still suppose that $d$ is invertible in our base field $\Bbbk$ and that a primitive $d$-th root of unity exists in $\Bbbk$. We then fix once again a primitive $d$-th root of unity $\zeta$. 

\subsection{Another presentation of the group algebra of $\mathcal{FB}_{d,n}$}
\label{sec:presentation-framed-idempotents}

We give another presentation of the group algebra, over the field $\Bbbk$, of the framed braid group. This is similar to \cite[Section 2.2]{JPA2}, where another presentation of the Yokonuma--Hecke is given, in terms of idempotents.

\begin{defn}
  An ordered partition of $n$ in $d$ parts is a $d$-tuple $(I_1,\ldots,I_d)$ of subsets of $\{1,\ldots,n\}$ such that
  \begin{enumerate}
  \item if $a\neq b$ then $I_a\cap I_b = \emptyset$,
  \item $I_1\cup\cdots\cup I_d = \{1,\ldots,n\}$.
  \end{enumerate}
  
  We denote by $\mathcal{P}_d(n)$ the set of ordered partitions of $n$ in $d$ parts.
\end{defn}

Note that a part $I_a$ is allowed to be empty and that the order of the sets $(I_1,\ldots,I_d)$ in such an ordered partition is relevant. A direct application of the multinomial theorem shows that $\lvert \mathcal{P}_d(n) \rvert = d^n$. Another explanation of this equality is that the set $\mathcal{P}_d(n)$ parametrizes the one-dimensional representations of the group $(\mathbb{Z}/d\mathbb{Z})^n$, as developed below. We define the position of $j$ in an ordered partition $I=(I_1,\ldots,I_d)$, denoted by $\pos_j(I)$, as the unique integer $1 \leq a \leq d$ such that $j \in I_a$.

Given an element $I\in \mathcal{P}_d(n)$, we define an element $E_I$ in the group algebra of the group $(\mathbb{Z}/d\mathbb{Z})^n$, that we also consider in the group algebra of the framed braid group $\mathcal{FB}_{d,n}$, by
\begin{equation}\label{def_EI}
  E_I = \prod_{i=1}^n\left(\frac{1}{d}\sum_{l=1}^d \zeta^{-l(\pos_i(I)-1)}t_i^l\right).
\end{equation}
The elements $(E_I)_{I\in \mathcal{P}_d(n)}$ form a complete family of mutually orthogonal minimal central idempotents in $\Bbbk[(\mathbb{Z}/d\mathbb{Z})^n]$. They satisfy:
\[
  t_iE_I=E_It_i=\zeta^{\pos_i(I)-1}E_I\ .
\]
Note that we can recover the element $t_i$ from the $E_I$'s:
\[
  t_i = \sum_{a=1}^d\zeta^{(a-1)}\sum_{\substack{I\in \mathcal{P}_d(n)\\ \pos_i(I) = a}} E_I.
\]

Since $\{t_1^{k_1}\cdots t_n^{k_n} \sigma \mid 1 \leq k_i \leq d, \sigma\in \mathcal{B}_n\}$ is a basis of $\Bbbk[\mathcal{FB}_{d,n}]$ we deduce that
\[
  \{E_I \sigma \mid I\in \mathcal{P}_{d}(n), \sigma\in \mathcal{B}_n\}
\]
is a basis of $\Bbbk[\mathcal{FB}_{d,n}]$. We also have immediately the following alternative presentation of the algebra $\Bbbk[\mathcal{FB}_{d,n}]$.

\begin{prop}\label{other_presentation_FB}
  The group algebra $\Bbbk[\mathcal{FB}_{d,n}]$ has a presentation with generators $s_1,\ldots,s_{n-1}$ and $E_I$, $I\in \mathcal{P}_d(n)$ with relations
  \begin{align*}
    s_is_j &= s_js_i, & \text{if }\lvert i-j \rvert >1,\\
    s_is_js_i &= s_js_is_j, & \text{if }\lvert i-j\rvert =1,\\
    E_IE_J &= \delta_{I,J}, & \text{for } I,J\in \mathcal{P}_n(d),\\
    \sum_{I\in \mathcal{P}_n(d)}E_I &= 1, &\\
    s_i E_I &= E_{s_i(I)}s_i, & \text{for } 1 \leq i < n \text{ and } I\in \mathcal{P}_d(n).
  \end{align*}
Here, $s_i(I)$ denotes the element of $\mathcal{P}_d(n)$ obtained from $I=(I_1,\ldots,I_d)$ by applying the transposition $(i, i+1)$ to each $I_k$.
\end{prop}

\subsection{Action of $\mathfrak{S}_d$ on $\Bbbk[\mathcal{FB}_{d,n}]$}
\label{sec:action-Sd}

The above presentation of the algebra $\Bbbk[\mathcal{FB}_{d,n}]$ makes apparent an action by automorphisms of the symmetric group $\mathfrak{S}_d$. Indeed the symmetric group $\mathfrak{S}_d$ acts on the set $\mathcal{P}_d(n)$ by 
\[
  w\cdot (I_1,\ldots,I_d) = (I_{w^{-1}(1)},\ldots,I_{w^{-1}(d)})\ \ \ \text{ for } w\in \mathfrak{S}_d\text{ and }I\in \mathcal{P}_d(n).
\]
This action naturally endows $\Bbbk[\mathcal{FB}_{d,n}]$ with an action of $\mathfrak{S}_d$ by linearly extending
\begin{equation}
  \label{eq:action_Sd}
  w\cdot\left (E_I \sigma\right) = E_{w\cdot I} \sigma\,,\ \ \ \text{ for }w\in \mathfrak{S}_d\,,\ I\in \mathcal{P}_d(n)\text{ and } \sigma\in \mathcal{B}_n.
\end{equation}
From the presentation in \cref{other_presentation_FB}, it follows easily that this induces an action of $\mathfrak{S}_d$ on $\Bbbk[\mathcal{FB}_{d,n}]$ by automorphisms of algebras (the only argument needed is that the action of $\mathfrak{S}_d$ on $\mathcal{P}_d(n)$ permuting the subsets commutes with the action of $\mathfrak{S}_n$ permuting the letters $1,\dots,n$).

\subsection{Fixed points subalgebra $\left(\Bbbk[\mathcal{FB}_{d,n}]\right)^{\mathfrak{S}_d}$}

The set of unordered partitions of $\{1,\ldots,n\}$ in at most $d$ parts is then in bijection with the set of orbits of $\mathcal{P}_d(n)$ under the action of the symmetric group $\mathfrak{S}_d$:
\[
  \mathcal{P}_d(n)/\mathfrak{S}_d \simeq \{ \{I_1,\ldots,I_d\} \mid I_1\cup\cdots\cup I_d =\{1,\ldots,n\} \text{ and } I_a\cap I_b = \emptyset \text{ for } 1 \leq a \neq b \leq d\}.
\] 
Once again, we stress that some of the $I_a$ might be empty. We denote by $B_{d}(n)$ the cardinal of $\mathcal{P}_d(n)/\mathfrak{S}_d$. It is easy to see that:
\begin{equation}\label{Bdn}
  B_d(n)=|\mathcal{P}_d(n)/\mathfrak{S}_d|=\sum_{\substack{\lambda\vdash n \\ \ell(\lambda)\leq d}}\binom{n}{\lambda}\frac{1}{l_1!\dots l_n!}\ .
\end{equation}
where $l_i$ in the sum is the number of parts of $\lambda$ equal to $i$, namely $\lambda=(1^{l_1},2^{l_2},\dots,n^{l_n})$. For $d\geq n$, the number $B_d(n)$ does not depend on $d$ any more and is equal to the Bell number.

Given $I\in \mathcal{P}_d(n)$ we denote by $[I]$ its orbit in $\mathcal{P}_d(n)$ under the action of $\mathfrak{S}_d$. For such an orbit $[I]$, we then define an element $E_{[I]}\in \Bbbk[\mathcal{FB}_{d,n}]$ fixed under the action of $\mathfrak{S}_d$:
\[
  E_{[I]} = \sum_{J\in [I]} E_J.
\]
From the formula \eqref{eq:action_Sd} giving the action of $\mathfrak{S}_d$ on the group algebra of $\mathcal{FB}_{d,n}$, it is immediate that the set 
\[
  \{E_{[I]}\sigma \mid [I]\in \mathcal{P}_d(n)/\mathfrak{S}_d, \sigma\in \mathcal{B}_n\}
\]
is a $\Bbbk$-basis of $\left(\Bbbk[\mathcal{FB}_{d,n}]\right)^{\mathfrak{S}_d}$. 

Finally, recall that we have introduced for $1 \leq i ,j \leq d$ the elements
\[
  E_{i,j}=\frac{1}{d}\sum_{a=1}^dt_i^at_{j}^{-a}\ \ \ \ \ \text{and}\ \ \ \ \ E_i=E_{i,i+1}\ .
\]
Multiplying $E_{i,j}$ by $\sum_{I\in \mathcal{P}_n(d)}E_I$ (which is $1$), we obtain that:
\[
  E_{i,j}=\sum_{\substack{I\in \mathcal{P}_n(d)\\ \pos_i(I) = \pos_j(I)}}E_{I}
\]
Moreover, one also easily checks (as in \cite[Lemma 4.1]{JPA2}) that
\[
  E_{[I]} = \prod_{\substack{1 \leq i,j \leq n\\\pos_i(I) = \pos_j(I)}}E_{i,j} \prod_{\substack{1 \leq i,j \leq n\\\pos_i(I) \neq \pos_j(I)}}(1-E_{i,j}).
\]
A statement similar to \cite[Proposition 4.2]{JPA2} follows straightforwardly.

\begin{prop}\label{prop-fixedpoints}
  The subalgebra $\left(\Bbbk[\mathcal{FB}_{d,n}]\right)^{\mathfrak{S}_d}$ of $\Bbbk[\mathcal{FB}_{d,n}]$ is the subalgebra generated by $s_1,\ldots,s_{n-1}$ and $E_1,\ldots,E_{n-1}$.
\end{prop}

\subsection{The tied braid algebra}
\label{sec:tied-braid-alg}

The tied braid monoid has been introduced by Aicardi and Juyumaya in \cite{AiJ}. In this section, we relate the algebra of the tied braid monoid to the subalgebra of fixed points $\left(\Bbbk[\mathcal{FB}_{d,n}]\right)^{\mathfrak{S}_d}$. This is similar to \cite[Section 4]{JPA2}, where a  relationship is obtained between the algebra of braids and ties and the fixed points of the Yokonuma--Hecke algebra under the action of $\mathfrak{S}_d$.

Let us start by defining the tied braid algebra as the algebra of the tied braid monoid.

\begin{defn}
  The tied braid algebra $\TB_n$ on $n$ strands is the $\Bbbk$-algebra with generators $\tilde{s}_1,\ldots,\tilde{s}_{n-1}$ that we require to be invertible and satisfying the usual braid relations, together with additional generators $\tilde{E}_1,\ldots,\tilde{E}_{n-1}$ satisfying the relations
  \begin{align*}
    \tilde{E}_i\tilde{E}_j &= \tilde{E}_j \tilde{E}_i \text{ for all } 1 \leq i,j < n,  & \tilde{E}_i^2&=\tilde{E}_i \text{ for all } 1 \leq i < n,\\
    \tilde{s}_i\tilde{E}_i &= \tilde{E}_i\tilde{s}_i \text{ for all } 1 \leq i < n, & \tilde{s}_i\tilde{E}_j &= \tilde{E}_j \tilde{s}_i \text{ for } \lvert i-j\rvert >1,\\
    \tilde{E}_i\tilde{s}_j\tilde{s}_i &= \tilde{s}_j\tilde{s}_i\tilde{E}_j \text{ for } \lvert i-j\rvert = 1, & \tilde{E}_i\tilde{s}_j\tilde{s}_i^{-1} &= \tilde{s}_j\tilde{s}_i^{-1}\tilde{E}_j \text{ for } \lvert i-j\rvert = 1,\\
    \tilde{E}_i\tilde{E}_j\tilde{s}_i &= \tilde{E}_j\tilde{s}_i\tilde{E}_j = \tilde{s}_i\tilde{E}_j\tilde{E}_i\text{ for } \lvert i-j\rvert = 1.
  \end{align*}
\end{defn}

Arcis and Juyumaya showed \cite[Proposition 4.8]{ArJ} that the tied braid monoid is a semidirect product between the partition monoid and the braid group. As a consequence, one can describe a basis of the tied braid algebra. Since the elements $\tilde{s}_i$ for $1 \leq i \leq n$ satisfy the braid relations, we have a well defined element $\tilde{\sigma}\in \TB_n$ corresponding to an element $\sigma\in \mathcal{B}_n$.

For $1 \leq i < j \leq n$ we set
\[
  \tilde{E}_{i,j} = \tilde{s}_{j-1}\cdots \tilde{s}_{i+1} \tilde{E}_i \tilde{s}_{i+1}^{-1}\cdots \tilde{s}_{j-1}^{-1}
\]
and also $\tilde{E}_{i,i}=1$ for $i\in\{1,\ldots,n\}$. Note that $\tilde{E}_{i,i+1} = \tilde{E}_i$. Finally, a set partition of $\{1,\dots,n\}$ having at most $n$ non-empty parts can be identified with an unordered partition $[I] \in \mathcal{P}_{n}(n)/\mathfrak{S}_n$ and we set
\[
  \tilde{f}_{[I]} = \prod_{\substack{1 \leq i < j \leq n\\ \pos_i(I) = \pos_j(I)}}\tilde{E}_{i,j}. 
\]
This element is an idempotent of $\TB_n$ and the description by Arcis--Juyumaya \cite{ArJ} of the tied monoid as a semidirect product immediately shows that
\[
  \{ \tilde{f}_{[I]}\tilde{\sigma} \mid [I]\in\mathcal{P}_n(n), \sigma\in \mathcal{B}_n \}
\]
is a basis of $\TB_n$. Therefore, $\TB_n$ is free over $\Bbbk[\mathcal{B}_n]$ of rank the usual Bell number $B_n(n)$. Moreover, a triangular change of basis easily shows that the elements $\tilde{f}_{[I]}$ can be replaced by the elements $\tilde{E}_{[I]}$ where
\[
  \tilde{E}_{[I]} = \prod_{\substack{1 \leq i,j \leq n\\\pos_i(I) = \pos_j(I)}}\tilde{E}_{i,j} \prod_{\substack{1 \leq i,j \leq n\\\pos_i(I) \neq \pos_j(I)}}(1-\tilde{E}_{i,j})\ .
\]
Similarly to \cite[Corollary 4.5]{JPA2}, we have:

\begin{thm}\label{thm_BT_fixedpoints}
  The following map defines a surjective morphism of algebras:
  \[
    \begin{array}{rcl}
      \TB_n & \to & \Bbbk[\mathcal{FB}_{d,n}]^{\mathfrak{S}_d} \\[0.4em]
      \tilde{s}_i & \mapsto & s_i\\[0.2em]
      \tilde{E}_i & \mapsto & E_i
    \end{array}
  \]
  which is an isomorphism if and only if $d\geq n$.
\end{thm}

\begin{proof}
  The assertion that the assignment $\tilde{s}_i \mapsto s_i$ and $\tilde{E}_i \mapsto E_i$ defines a morphism of algebras boils down to a simple calculation using the relations in $\Bbbk[\mathcal{FB}_{d,n}]$. The surjectivity follows from \cref{prop-fixedpoints}. If $d\geq n$, en element $[I] \in \mathcal{P}_{d}(n)/\mathfrak{S}_d$ can be identified with an unordered partition in $\mathcal{P}_{n}(n)/\mathfrak{S}_n$ (by removing some empty subsets in $[I]$). In this case, one has immediately that a basis of $\TB_n$ is sent to a basis of $\Bbbk[\mathcal{FB}_{d,n}]^{\mathfrak{S}_d}$, and the morphism therefore becomes an isomorphism.

  If $d<n$, elements $[I] \in \mathcal{P}_{d}(n)/\mathfrak{S}_d$ are identified with a strictly smaller subset of $\mathcal{P}_{n}(n)/\mathfrak{S}_n$ (the set partitions with at most $d$ non-empty parts). In this case, a strictly smaller subset of the basis of $\TB_n$ is sent onto the basis of $\Bbbk[\mathcal{FB}_{d,n}]^{\mathfrak{S}_d}$ and the morphism cannot be injective.
\end{proof}

\section{Tied braid algebra and centralizers of tensor products}
\label{sec:tied-braid_centralizers}

We now study the action of the tied braid algebra on centralizers, in the spirit of \cref{sec:dble_grp}. We start from the construction of \cref{sec:framedbraid}, with the additional assumption that all bialgebras $A_1,\dots,A_d$ are equal to one and the same $\Bbbk$-bialgebra $A^{(0)}$, and all modules $V_1,\dots,V_d$ are equal to one and the same $A^{(0)}$-module $V^{(0)}$: 
\[
  A=A^{(0)}\boxtimes \cdots \boxtimes A^{(0)}\ \ \ \ \text{and}\ \ \ \ V_1=\dots=V_d=V^{(0)}\ .
\]
So we have an algebra morphism from the braid group algebra to the centralizer of $A^{(0)}$ on $(V^{(0)})^{\otimes n}$ (for any $n\geq 0$):
\begin{equation}\label{braid_in_tA}
  \phi^{(0)}\ \ \ :\ \ \  \Bbbk\mathcal{B}_n \to \End_{A^{(0)}}\bigl((V^{(0)})^{\otimes n}\bigr)\ .
\end{equation}
Note that in this section this single morphism plays the role of the morphisms $\phi_b$, for all $b=1,\dots,d$. This allowed us to define, in \cref{thm:action-framed}, a morphism of algebras 
\begin{equation}\label{braid_in_A_bis}
  \Phi\ \ \ :\ \ \  \Bbbk[\mathcal{FB}_{d,n}] \rightarrow \End_{A}(V^{\otimes n})\ .
\end{equation}
This morphism was given explicitly on the generators $t_i,s_i$ of $\mathcal{FB}_{d,n}$. Below we will also give it on the new generators $E_I$, defined in \eqref{def_EI}, of $\Bbbk[\mathcal{FB}_{d,n}]$.

\subsection{Centralizers of semi-direct product.}

As explained in \cref{sec:alg_prel}, we have the natural action of $\mathfrak{S}_d$ on $A$ (permuting the factors) and the corresponding algebra $A\rtimes \mathfrak{S}_d$ acting on $V^{\otimes n}$. Moreover, we recall that:
\[
  \End_{A\rtimes \mathfrak{S}_d}(V^{\otimes n})= \End_A(V^{\otimes n})^{\mathfrak{S}_d}\ .
\]
The action of $\mathfrak{S}_d$ on the centralizer $\End_A(V^{\otimes n})$ is simply by conjugating with the action of $\mathfrak{S}_d$ on $V^{\otimes n}$. We recall for convenience of the reader that we have the following decomposition of $V^{\otimes n}$:
\begin{equation}\label{dec-Votimesn}
  V^{\otimes n}=\bigoplus_{a_1,\dots,a_n=1}^dV_{a_1}\otimes\dots\otimes V_{a_n}\ .
\end{equation}
and that an element $\sigma\in\mathfrak{S}_d$ acts by permuting the summands as
\begin{equation}\label{act_Sd}
  \sigma\ :\ V_{a_1}\otimes\dots\otimes V_{a_n}\ \to\ V_{\sigma(a_1)}\otimes\dots\otimes V_{\sigma(a_n)}\ .
\end{equation}
Note that this is different from the action of $\mathfrak{S}_n$ by permuting the factors.

\subsection{Action of the tied braid algebra.}
\label{sec:action-tied-centralizers}

The image of the element $E_I$ under the morphism $\Phi$ in \eqref{braid_in_A_bis} is easily seen to be the projector onto one summand of the decomposition \eqref{dec-Votimesn} of $V^{\otimes n}$:
\begin{equation}\label{act:EI}
  \Phi(E_I)\ :\ V^{\otimes n}\ \to\ V_{\pos_1(I)}\otimes \dots\otimes V_{\pos_n(I)}\ .
\end{equation}
Indeed, recall that $E_I$ is the idempotent associated to the irreducible representation of $(\mathbb{Z}/d\mathbb{Z})^n$ corresponding  to $t_i\mapsto \zeta^{\pos_i(I)-1}$. In view of the action of $\Phi(t_i)$ in \cref{thm:action-framed}, the above description of $\Phi(E_I)$ follows immediately.

We are ready to state the general result relating the tied braid algebra with centralizers of tensor products. Note that the images by $\Phi$ of the elements $s_i$ and $E_i$ of the framed braid group algebra were described explicitly in \cref{sec:framedbraid}.
\begin{thm}
  \label{thm:BT_centralizer}
  For any $d\geq 1$, we have a morphism of algebras
  \begin{equation}
    \Psi\ :\
    \begin{array}{rcl}
      \TB_n & \rightarrow &  \End_{A\rtimes \mathfrak{S}_d}(V^{\otimes n})\\[0.4em]
      \tilde{s}_i & \mapsto & \Phi(s_i)\ \ \ \ (i=1,\dots,n-1)\,,\\[0.4em]
      \tilde{E}_i & \mapsto & \Phi(E_i)\ \ \ \ (i=1,\dots,n-1)\,.
    \end{array}\label{eq:mor_tb}
  \end{equation}
  This morphism is surjective if the morphism $\Phi$ in \eqref{braid_in_A_bis} is surjective.
\end{thm}

\begin{proof}
  We start by checking that the morphism $\Phi$ in \eqref{braid_in_A_bis} is $\mathfrak{S}_d$-equivariant and thus induces an algebra morphism
  \begin{equation}\label{eq:mor_fixedpoints}
    \Bbbk[\mathcal{FB}_{d,n}]^{\mathfrak{S}_d}\rightarrow \End_{A}(V^{\otimes n})^{\mathfrak{S}_d}\ .
  \end{equation}
  The equivariance of the morphism $\Phi$ on the generators $s_i$ amounts to the fact that the action of $\mathfrak{S}_d$ on $V^{\otimes n}$ commutes with the image $\Phi(s_i)$. This follows easily from the description of $\Phi(s_i)=\sigma_i$ in \cref{sec:framedbraid}. Indeed either $\Phi(s_i)$ sends a summand $\dots V_{a_i}\otimes V_{a_{i+1}}\dots$ to $\dots V_{a_{i+1}}\otimes V_{a_i}\dots$, in which case it commutes with the action \eqref{act_Sd}, or it acts inside a given summand $\dots V_{a_i}\otimes V_{a_{i+1}}\dots$ (if $a_i=a_{i+1}$) in which case its action does not depend on the value of $a_i$, since all algebras $A_b$ are the same, and it also commutes with the action \eqref{act_Sd}.

  For the equivariance of the morphism $\Phi$ on the generators $E_I$, we have that $\Phi(\sigma\cdot E_I)=\Phi(E_{\sigma\cdot I})$ acts as the projector on the summand $V_{\pos_1(\sigma\cdot I)}\otimes \dots\otimes V_{\pos_n(\sigma\cdot I)}$. Since $\pos_j(\sigma\cdot I)=\sigma(\pos_j(I))$, it is equal to the conjugate by the action of $\sigma$ in \eqref{act_Sd} of the projector on $V_{\pos_1(I)}\otimes \dots\otimes V_{\pos_n(I)}$. This concludes the proof of the equivariance.

  Composing the morphism in \eqref{eq:mor_fixedpoints} with the one obtained in \cref{thm_BT_fixedpoints} sending $\TB_n$ to $\Bbbk[\mathcal{FB}_{d,n}]^{\mathfrak{S}_d}$, we get the morphism $\Psi$
  \[
    \Psi\ :\ \TB_n\rightarrow \End_{A}(V^{\otimes n})^{\mathfrak{S}_d}=\End_{A\rtimes \mathfrak{S}_d}(V^{\otimes n})\ ,
  \]
  the last equality being in \cref{lem:centralizer_semi}. The statement about the surjectivity of the morphism follows from the general fact in \cref{lem:centralizer_semi}.
\end{proof}

\subsection{Examples of algebras of braid and ties}

In this section, we take an indeterminate $q$ and the field $\Bbbk=\mathbb{C}(q)$.

\subsubsection{The Hecke algebra of braids and ties}

We take $A^{(0)} = U_q(\mathfrak{gl}_N)$ and $V^{(0)}$ the vector representation of $U_q(\mathfrak{gl}_N)$ of dimension $N$. Then the algebra $A$ is isomorphic to $U_q(\mathfrak{gl}_N^{\oplus d})$ and the representation $V$ is the vector representation of dimension $dN$.

The Hecke algebra of braids and ties $\BT^{\Hec}_n$ is the quotient of the tied braid algebra $\TB_n$ by the additional quadratic relation  of the Yokonuma--Hecke algebra
\[
  \tilde{s}_i^2 = (q-q^{-1})\tilde{s}_i\tilde{E}_i + 1.
\]
This algebra was defined in \cite{AJ-BT} but we have slightly modified the name.

\begin{thm}\label{thm:BT-H}
  We have a surjective morphism
  \[
    \Psi\ :\ \BT^{\Hec}_n \rightarrow \End_{U_q(\mathfrak{gl}_N^{\oplus d})\rtimes \mathfrak{S}_d}(V^{\otimes n})\ . 
  \]
  This is an isomorphism if and only if $n\leq d$ and $n\leq N$.
\end{thm}

\begin{proof}
  We already have a surjective morphism $\TB_n \rightarrow \End_{U_q(\mathfrak{gl}_n^{\oplus d})\rtimes \mathfrak{S}_d}(V^{\otimes n})$ thanks to \cref{thm:BT_centralizer}. It then suffices to recall that the quadratic relation is satisfied in the centralizer as shown in \cref{thm_YH}.

  Then we have to show that $\Psi$ is injective if and only if $n\leq d$ and $n\leq N$. One can forget about the $U_q(\mathfrak{gl}_N^{\oplus d})\rtimes \mathfrak{S}_d$-module structure and consider only the resulting map to $\End_{\Bbbk}(V^{\otimes n})$. Then the injectivity follows from the results of \cite{RH}.
\end{proof}

Here we take $d=n$ and we use the above theorem for $N\geq n$. Applying the general results from \cref{sec:dble_grp}, we deduce the following isomorphism
\begin{equation}
  \BT^{\Hec}_n\cong \bigoplus_{\lambda\vdash n} \Mat_{\binom{n}{\lambda}/l_1!\cdots l_n!}\left(\mathcal{C}_{\mathfrak{S}_{l_1}\times\dots\times \mathfrak{S}_{l_n}}\left(\Hec_{\lambda_1}\otimes\cdots\otimes \Hec_{\lambda_n}\right)\right),
\end{equation}
where we recall that a partition $\lambda$ in the sum is written as $(1^{l_1},2^{l_2},\dots,n^{l_n})$, that is, $l_i$ is the number of $i$ occurring in $\lambda$. The algebras $\Hec_{\lambda_i}$ are usual Hecke algebras. We recover, with a slightly different formulation, an isomorphism theorem proved in \cite{ER}.

Specifying the setting of \cref{sec:consequencesfixedpoints}, we find an indexation of irreducible representations over $\mathbb{C}(q)$ of $\BT^{\Hec}_n$ by
\[
  (\lambda,\rho_1,\dots,\rho_n,\Lambda_1,\dots,\Lambda_n)\,,\ \ \ \text{where}\ \lambda\vdash n\,,\ \ \rho_i\vdash \lambda_i\,,\ \ \Lambda_i \vdash l_i\,,
\]
where we have identified the irreducible representations of the symmetric group and of the Hecke algebra with partitions. The dimension of this representation is:
\[
  \binom{n}{\lambda}\frac{\prod_{i=1}^n\dim\rho_i\dim\Lambda_i}{l_1!\dots l_n!}\ .
\]
We recover the description of the semisimple representation theory of $\BT^{\Hec}_n$, see \cite{JPA2,RH11}.

Note that the general formula for the total dimension of the direct sum of matrix algebras simplifies nicely and we find:
\[
  \dim\BT^{\Hec}_n=\sum_{\lambda\vdash n}\binom{n}{\lambda}^2\frac{\lambda_1!\dots\lambda_n!}{l_1!\dots l_n!}=n!\sum_{\lambda\vdash n}\binom{n}{\lambda}\frac{1}{l_1!\dots l_n!}=n!B(n)\ .
\]
where $B(n)$ is the Bell number. The sequence $n!B(n)$ starting from $n=0$ with $1,1,4,30,360$ is the sequence A137341 on \cite{OEIS}.

\subsubsection{The Temperley--Lieb algebra of braids and ties.}

We keep the same setting and we impose moreover that $N=2$. In view of \cref{sec:framed-TL}, it is natural to define the Temperley--Lieb algebra of braids and ties $\BT^{\TL}_n$ as the quotient of the Hecke algebra of braids and ties $\BT^{\Hec}_n$ by the additional relation:
\[
  \tilde{E}_1\tilde{E}_2(1-q^{-1}\tilde{s}_1-q^{-1}\tilde{s}_2+q^{-2}\tilde{s}_1\tilde{s}_2+q^{-2}\tilde{s}_2\tilde{s}_1-q^{-3}\tilde{s}_1\tilde{s}_2\tilde{s}_1)=0\ .
\]

\begin{thm}\label{thm:BT-TL}
  We have a surjective morphism
  \[
    \Psi\ :\ \BT^{\TL}_n \rightarrow \End_{U_q(\mathfrak{gl}_2^{\oplus d})\rtimes \mathfrak{S}_d}(V^{\otimes n})\ . 
  \]
  This is an isomorphism if and only if $n\leq d$.
\end{thm}

\begin{proof}
  We already have a surjective morphism in \cref{thm:BT-H} from the Hecke algebra of braids and ties $\BT^{\Hec}_n$. Moreover, the fact that the additional relation defining $\BT^{\TL}_n$ is satisfied in the centralizer was already proved in \cref{thm_FTL}. As for the preceding theorem, the injectivity statement follows from the results of \cite{RH}. We skip the details.
\end{proof}

Applying the general results from \cref{sec:dble_grp} (and taking $d=n$), we obtain the following isomorphism in terms of usual Temperley--Lieb algebras:
\begin{equation}
  \BT^{\TL}_n\cong \bigoplus_{\lambda\vdash n} \Mat_{\binom{n}{\lambda}/l_1!\cdots l_n!}\left(\mathcal{C}_{\mathfrak{S}_{l_1}\times\dots\times \mathfrak{S}_{l_n}}\left(\TL_{\lambda_1}\otimes\cdots\otimes \TL_{\lambda_n}\right)\right)\ .
\end{equation}
Specifying the setting of \cref{sec:consequencesfixedpoints}, we find an indexation of irreducible representations over $\mathbb{C}(q)$ of $\BT^{\TL}_n$ by
\[
  (\lambda,\rho_1,\dots,\rho_n,\Lambda_1,\dots,\Lambda_n)\,,\ \ \ \text{where}\ \lambda\vdash n\,,\ \ \rho_i\vdash \lambda_i\ (\ell(\rho_i)\leq 2)\,,\ \ \Lambda_i \vdash l_i\,,
\]
where we have identified the irreducible representation of $\TL_{k}$ with partitions of $k$ with no more than two parts. The dimension of this representation is the same as for the algebra $\BT^{\Hec}_n$. The total dimension is:
\[
  \dim\BT^{\TL}_n=\sum_{\lambda\vdash n}\binom{n}{\lambda}^2\frac{C_{\lambda_1}\dots C_{\lambda_n}}{l_1!\dots l_n!}\ ,
\]
where $C_k=\frac{1}{k+1}\binom{2k}{k}$ is the Catalan number. This sequence starts (from $n=0$) with $1,1,4,29,334,5512$ and does not seem to be on \cite{OEIS}. The algebra $\BT^{\TL}_n$ is also called partition Temperley--Lieb algebra, see \cite[Section 5]{RH} and \cite{Juy2}.

\subsubsection{The BMW algebra of braids and ties.}

Here we extend our ground field with another indeterminate $\Bbbk=\mathbb{C}(q,a)$. We define the BMW algebra of braids and ties $\BT^{\BMW}_n$ as the quotient of the tied braids algebra $\TB_n$ by the relations:
\begin{align}
  e_i \tilde{s}_i & =  a^{-1} e_i & & \mbox{for} \ i=1,\dots, n-1\,, \label{BTBMW1}\\
  e_i \tilde{E}_i & =  e_i & & \mbox{for} \ i=1,\dots, n-1\,, \label{BTBMWeE}\\
  e_i \tilde{s}_{j}^{\pm1} e_i \tilde{E}_{j} & = a^{\pm1} e_{i}\tilde{E}_{j} & & \mbox{for} \ |i-j|=1\,, \label{BTBMW2}
\end{align}
where $e_i$ is defined as $e_i = \tilde{E}_i - \frac{\tilde{s}_i-\tilde{s}_i^{-1}}{q - q^{-1}}$. 

Note that Relations \eqref{BTBMW1}) and \eqref{BTBMW2} were defining relations of the framization of the BMW algebra $\FBMW_{d,n}$. On the other hand, Relation \eqref{BTBMWeE} was a consequence of the defining relations in $\FBMW_{d,n}$. However, it was proved using the explicit definition of the element $E_i$ in terms of $t_i,t_{i+1}$. Such an argument is not available here and that is why we put \eqref{BTBMWeE} as a defining relation of $\BT^{\BMW}_n$. And indeed one can check for $n=2$ that it is not implied by the other relations. It is easy to check that all other relations from \cref{lem:FBMW} are also satisfied in $\BT^{\BMW}_n$. 

In fact our goal is to have defining relations for $\BT^{\BMW}_n$ which are enough to prove the expected isomorphism:
\[
  \BT^{\BMW}_n\cong \bigoplus_{\lambda\vdash n} \Mat_{\binom{n}{\lambda}/l_1!\cdots l_n!}\left(\mathcal{C}_{\mathfrak{S}_{l_1}\times\dots\times \mathfrak{S}_{l_n}}\left(\BMW_{\lambda_1}\otimes\cdots\otimes \BMW_{\lambda_n}\right)\right)
\]
and in particular, to lead to a resulting dimension of $\BT^{\BMW}_n$ given by
\[
  \dim\BT^{\BMW}_n=\sum_{\lambda\vdash n}\binom{n}{\lambda}^2\frac{(2\lambda_1)!\dots (2\lambda_n)!}{2^nl_1!\dots l_n!\lambda_1!\dots\lambda_n!}\ .
\]
This sequence of dimensions starts (from $n=0$) with $1,1,5,48,747$ and is not on \cite{OEIS}. We have checked that the defining relations above give the correct dimension for $n\leq 3$.

For some specializations of $a$, we leave to the reader to formulate the obvious analogues of \cref{prop:FBMW} (adding the $\mathfrak{S}_d$-action into the picture, as above).


\bibliographystyle{habbrv}
\bibliography{biblio}

\end{document}